\documentclass[12pt]{amsart}
\pdfoutput=1
\usepackage{amsmath, amsthm, amssymb}
 \usepackage{ifpdf}
\ifpdf
\usepackage[pdftex]{graphicx}
\else
\usepackage[dvips]{graphicx}
\fi
\usepackage{tikz}
 	 \usetikzlibrary{arrows,backgrounds} 
\usepackage[all]{xy}

\usepackage[pdftex,plainpages=false,hypertexnames=false,pdfpagelabels]{hyperref}
\newcommand{\arXiv}[1]{\href{http://arxiv.org/abs/#1}{\tt arXiv:\nolinkurl{#1}}}

\newcommand{\googlebooks}[1]{(preview at \href{http://books.google.com/books?id=#1}{google books})}
\theoremstyle{plain}
\newtheorem{thm}{Theorem}[section]
\newtheorem{cor}[thm]{Corollary}
\newtheorem{lem}[thm]{Lemma}

\newtheorem{prop}[thm]{Proposition}

\theoremstyle{definition}
\newtheorem{defn}[thm]{Definition}

\newtheorem{nota}[thm]{Notation}
\newtheorem{note}[thm]{Note}

\newtheorem{ex}[thm]{Example}
\newtheorem{rem}[thm]{Remark}
\newtheorem{rems}[thm]{Remarks} 
\newtheorem{alg}[thm]{Algorithm}

\DeclareMathOperator{\coker}{coker}
\DeclareMathOperator{\capind}{capind}
\DeclareMathOperator{\caps}{caps}
\DeclareMathOperator{\cupind}{cupind}
\DeclareMathOperator{\cups}{cups}
\DeclareMathOperator{\id}{id}
\DeclareMathOperator{\im}{im}
\DeclareMathOperator{\ind}{ind}
\DeclareMathOperator{\Mor}{Mor}
\DeclareMathOperator{\Ob}{Ob}
\DeclareMathOperator{\rel}{rel}
\DeclareMathOperator{\Tot}{Tot}
\DeclareMathOperator{\ts}{ts}
\newcommand{\D}{\displaystyle}
\newcommand{\comment}[1]{}
\newcommand{\hs}{\hspace{.07in}}
\newcommand{\hsp}[1]{\hs\text{#1}\hs}
\newcommand{\be}{\begin{enumerate}}
\newcommand{\ee}{\end{enumerate}}
\newcommand{\itm}[1]{\item[\underline{\ensuremath{#1}:}]} 
\newcommand{\itt}[1]{\item[\underline{\text{#1}:}]} 
\newcommand{\N}{\mathbb{N}}

\newcommand{\C}{\mathbb{C}}

\newcommand{\Z}{\mathbb{Z}} 
\newcommand{\W}{\mathcal{W}} 
\newcommand{\T}{\mathcal{T}} 
 
\newcommand{\set}[2]{\left\{#1 \big| #2\right\}}

\newcommand{\thh}{^{\text{th}}}
\newcommand{\A}{\mathcal{A}}
\newcommand{\B}{\mathcal{B}}
\newcommand{\Atl}{{\sf{Atl}}}
\newcommand{\AAA}{{\sf{A}}}
\newcommand{\BB}{{\sf{B}}}
\newcommand{\CC}{{\sf{C}}}

\newcommand{\TT}{{\sf{T}}}
\newcommand{\PO}{{\sf{PO}}}
\newcommand{\QQ}{{\sf{Q}}}
\newcommand{\cAtl}{{\sf{cAtl}}} 
\newcommand{\sAtl}{{\sf{sAtl}}}
\newcommand{\ssAtl}{{\sf{ssAtl}}} 
\newcommand{\aD}{{\sf{a}\Delta}}
\newcommand{\tlD}{{\sf{tl}\Delta}}
\newcommand{\ccD}{{\sf{c}\Delta}}
\newcommand{\sD}{{\sf{s}\Delta}}
\newcommand{\ssD}{{\sf{ss}\Delta}}
\newcommand{\TL}{{\sf{TL}}}
\newcommand{\sTL}{{\sf{sTL}}}
\newcommand{\ssTL}{{\sf{ssTL}}}
\newcommand{\aTL}{{\sf{aTL}}}

\newcommand{\Cat}{\sf{Cat}}
\newcommand{\Set}{\sf{Set}}

\newcommand{\op}{^{\sf{op}}}   

\input xy
\xyoption{all}

\begin{document}
 \tikzstyle{shaded}=[fill=red!10!blue!20!gray!30!white]
\tikzstyle{shaded line}=[double=red!10!blue!20!gray!30!white, double distance=1.5mm, draw=black]
\tikzstyle{unshaded}=[fill=white]
\tikzstyle{unshaded line}=[double=white, double distance=1.5mm, draw=black]
\tikzstyle{Tbox}=[circle, draw, thick, fill=white, opaque,]
\tikzstyle{empty box}=[circle, draw, thick, fill=white, opaque, inner sep=2mm]
\tikzstyle{background rectangle}= [fill=red!10!blue!20!gray!40!white,rounded corners=2mm] 
\tikzstyle{on}=[very thick, red!50!blue!50!black]
\tikzstyle{off}=[gray]

\tikzstyle{traces}=[scale=.2, inner sep=1mm]
\tikzstyle{quadratic}=[scale=.4, inner sep=1mm, baseline]
\tikzstyle{annular}=[scale=.7, inner sep=1mm, baseline]
\tikzstyle{rectangular}=[scale=.7, inner sep=1mm, baseline]
\tikzstyle{make triple edge size}= [scale=.4, inner sep=1mm,baseline] 
\tikzstyle{icosahedron network}=[scale=.3, inner sep=1mm, baseline]
\tikzstyle{ATLsix}=[scale=.25, baseline]
\tikzstyle{TL12}=[scale=.15,baseline]
\tikzstyle{PAdefn}=[scale=.7,baseline]
\tikzstyle{TLEG}=[scale=.5,baseline]

\title[A Cyclic Approach to the  Annular Temperley-Lieb Category]
 {A Cyclic Approach to the Annular Temperley-Lieb Category}

\author{ David Penneys}

\address{Department of Mathematics, University of California, Berkeley, 94720}

\email{dpenneys@math.berkeley.edu}


\date{\today}


\begin{abstract}
In \cite{MR1865703}, Jones found two copies of the cyclic category $\ccD$ in the annular Temperley-Lieb category $\Atl$. We give an abstract presentation of $\Atl$ to discuss how these two copies of $\ccD$ generate $\Atl$ together with the coupling constants and the coupling relations. We then discuss modules over the annular category and homologies of such modules, the latter of which arises from the cyclic viewpoint.
\end{abstract}

\maketitle
\tableofcontents

\section{Introduction}
The Temperley-Lieb algebras have been studied extensively beginning with Temperley and Lieb's first paper in statistical mechanics regarding hydrogen bonds in ice-type lattices \cite{MR0498284}.  Since, these algebras have been instrumental in many areas of mathematics, including subfactors \cite{MR696688} and knot theory \cite{MR0766964}. The well known diagrammatic representation of these algebras was introduced by Kauffman in \cite{MR899057} in his skein theoretic definition of the Jones polynomial. From these diagrams, we get the Temperley-Lieb category whose objects are $n$ points on a line, morphisms are diagrams with non-intersecting strings, and composition is stacking tangles vertically (we read bottom to top). 

Historically, the (affine/annular) Temperley-Lieb algebras have been presented as quotients of the (affine) Hecke algegras \cite{MR1309131}. Graham and Lehrer define cellular structures for these algebras in \cite{MR1376244}, and they give the representation theory for affine Temperley-Lieb in \cite{MR1659204}. Jones' definition of the annular Temperley-Lieb category  (see \cite{math/9909027}, \cite{MR1929335}), which we will denote $\Atl$, differs slightly Graham and Lehrer's. First, $\Atl$-tangles have a checkerboard shading, so each disk has an even number of boundary points. Second, the rotation is periodic in $\Atl$, similar to the rotation in Connes' cyclic category $\ccD$, studied by Connes \cite{MR0777584}, \cite{MR1303779}, Loday and Quillen \cite{MR0695381}, \cite{MR1217970}, and Tsygan \cite{MR0695483}. Jones found a connection between $\Atl$ and $\ccD$ in \cite{MR1865703}, and raised the question we now address: how does $\Atl$ arise from the interaction of two copies of the cyclic category?

In answering this question, we see $\Atl$ evolve from simple categories. The opposite of the simplicial category $\sD\op$ (see \ref{simplicial}) has a well known pictorial representation much like the Temperley-Lieb category: objects are $2n+2$ points on a line, morphisms are rectangular planar tangles with only shaded caps and unshaded cups, and composition is stacking. In fact, these diagrams closely resemble the string diagrams arising from an adjoint functor pair.

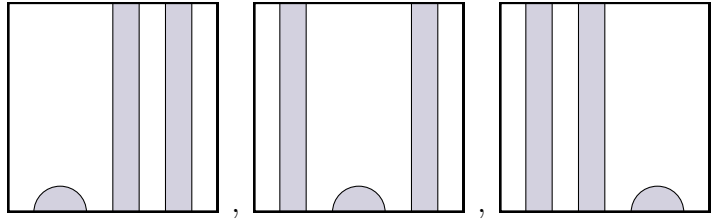
\begin{figure}[!ht]
\begin{tikzpicture}[rectangular]
	\clip (0,0) rectangle (4,4);
	
	\filldraw[shaded] (2,-1) -- (2,5) -- (2.5, 5)  -- (2.5,-1); 
	\filldraw[shaded] (3,-1) -- (3,5) -- (3.5, 5)  -- (3.5,-1); 
	\filldraw[shaded] (.5,0) arc (180:0:5mm);	

	\draw[ultra thick] (0,0) rectangle (4,4);
\end{tikzpicture}\hs,\hs
\begin{tikzpicture}[rectangular]
	\clip (0,0) rectangle (4,4);
	
	\filldraw[shaded] (.5,-1) -- (.5,5) -- (1, 5)  -- (1,-1); 
	\filldraw[shaded] (3,-1) -- (3,5) -- (3.5, 5)  -- (3.5,-1); 
	\filldraw[shaded] (1.5,0) arc (180:0:5mm);	

	\draw[ultra thick] (0,0) rectangle (4,4);
\end{tikzpicture}\hs,\hs
\begin{tikzpicture}[rectangular]
	\clip (0,0) rectangle (4,4);
	
	\filldraw[shaded] (.5,-1) -- (.5,5) -- (1, 5)  -- (1,-1); 
	\filldraw[shaded] (1.5,-1) -- (1.5,5) -- (2, 5)  -- (2,-1); 
	\filldraw[shaded] (2.5,0) arc (180:0:5mm);	

	\draw[ultra thick] (0,0) rectangle (4,4);
\end{tikzpicture}
\caption{Face maps $d_0,d_1,d_2\colon[2]\to[1]$}
\end{figure}

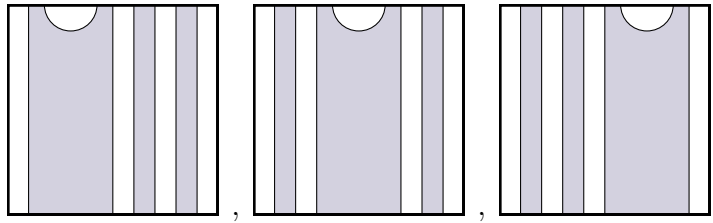
\begin{figure}[!ht]
\begin{tikzpicture}[rectangular]
	\clip (0,0) rectangle (4,4);
	
	\filldraw[shaded] (.4,-1) -- (.4,5) -- (2, 5)  -- (2,-1); 
	\filldraw[shaded] (2.4,-1) -- (2.4,5) -- (2.8, 5)  -- (2.8,-1);
	\filldraw[shaded] (3.2,-1) -- (3.2,5) -- (3.6, 5)  -- (3.6,-1);

	\filldraw[unshaded] (.7,4) arc (-180:0:5mm);			
	\draw[ultra thick] (0,0) rectangle (4,4);
\end{tikzpicture}\hs,\hs
\begin{tikzpicture}[rectangular]
	\clip (0,0) rectangle (4,4);
	
	\filldraw[shaded] (.4,-1) -- (.4,5) -- (.8, 5)  -- (.8,-1);
	\filldraw[shaded] (1.2,-1) -- (1.2,5) -- (2.8, 5)  -- (2.8,-1); 
	\filldraw[shaded] (3.2,-1) -- (3.2,5) -- (3.6, 5)  -- (3.6,-1);

	\filldraw[unshaded] (1.5,4) arc (-180:0:5mm);			
	\draw[ultra thick] (0,0) rectangle (4,4);
\end{tikzpicture}\hs,\hs
\begin{tikzpicture}[rectangular]
	\clip (0,0) rectangle (4,4);
	
	\filldraw[shaded] (.4,-1) -- (.4,5) -- (.8, 5)  -- (.8,-1);
	\filldraw[shaded] (1.2,-1) -- (1.2,5) -- (1.6, 5)  -- (1.6,-1);
	\filldraw[shaded] (2,-1) -- (2,5) -- (3.6, 5)  -- (3.6,-1);

	\filldraw[unshaded] (2.3,4) arc (-180:0:5mm);			
	\draw[ultra thick] (0,0) rectangle (4,4);
\end{tikzpicture}
\caption{Degeneracies $s_0,s_1,s_2\colon[2]\to[3]$}
\end{figure}

An asymmetry is present in the above tangles: all shaded regions can be ``capped" by applying a face map, but not every unshaded region can be ``cupped" by applying a degeneracy. This asymmetry can be corrected by closing the rectangular tangles into annuli, still enforcing the same shading requirements. Jones showed the resulting category is isomorphic to $\ccD\op$ in \cite{MR1865703}. Of course the category with the reverse shading is also isomorphic to $\ccD$ (and $\ccD\op$), and these two subcategories generate $\Atl$.

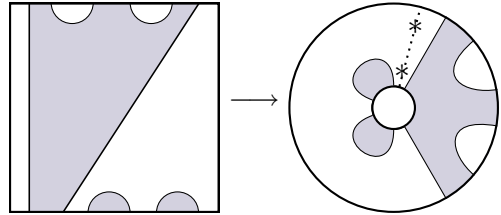
\begin{figure}[!ht]
$$
\begin{tikzpicture}[rectangular]
	\clip (-2,-2) rectangle (2,2);
	
	\draw[very thick] (-1.6,-2) -- (-1.6,2);
	\draw[very thick] (-1,-2) -- (1.6,2);
	\fill[shaded] (-1.6,-2) -- (-1.6,2) -- (1.6, 2)  -- (-1,-2); 
	\filldraw[shaded] (-.5,-2) arc (180:0:4mm);	
	\filldraw[shaded] (.8,-2) arc (180:0:4mm);	
	\filldraw[unshaded] (-1.2,2) arc (-180:0:4mm);	
	\filldraw[unshaded] (.3,2) arc (-180:0:4mm);
	\draw[ultra thick] (-2,-2) rectangle (2,2);
\end{tikzpicture}
\longrightarrow
\begin{tikzpicture}[annular]
	\clip (0,0) circle (2cm);
	\node at (0,0)  [empty box] (T) {};
	\node at (75:1.60cm) {$*$};
	\filldraw[shaded] (-60:4cm)--(0,0)--(60:4cm);	
	\filldraw[shaded] (T.90) .. controls ++(90:11mm) and ++           (160:11mm) .. (T.160);
	\filldraw[shaded] (T.270) .. controls ++(270:11mm) and ++           (200:11mm) .. (T.200);
	\filldraw[unshaded] (-40:3cm)--(-40:2cm) .. controls ++(140:6mm) and ++(170:11mm) .. (-10:2cm) -- (-10:3cm);
	\filldraw[unshaded] (10:3cm)--(10:2cm) .. controls ++(190:6mm) and ++(220:11mm) .. (40:2cm) -- (40:3cm);
	\draw[thick,dotted] (75:0cm) -- (75:2cm);
	\node at (0,0)  [empty box] (T) {};
	\node at (T.68) [above] {$*$};
	\draw[ultra thick] (0,0) circle (2cm);	
\end{tikzpicture}
$$
\caption{Closing up rectangular tangles into annuli}
\end{figure}

\subsection{Outline}
In Section \ref{Atl}, we will define $\Atl$ and offer candidates for generators and relations. We will then prove some uniqueness results which will be crucial to our approach. In Section \ref{aD}, we will take these candidates and define an abstract category $\aD$, the annular category, via generators and relations. We then prove existence of a standard form for words. In Section \ref{isocat}, we prove Theorem \ref{iso}, which says there is an isomorphism of involutive categories $\Atl\cong\aD$ (the isomorphism preserves an involution).

After we have our description of $\Atl$ in terms of abstract generators and relations, we recover the result of Jones in \cite{MR1865703} in \ref{cyclic}, i.e. two isomorphisms from $\ccD\op$ to subcategories $\cAtl^\pm$ of $\Atl$. After a note on augmentation of the cyclic category in \ref{augment}, we prove the main result of the paper, Theorem \ref{pushoutAtl}, which shows $\Atl$ is a quotient of the pushout of augmented copies of $\ccD$ and $\ccD\op$ over a groupoid $\TT$ of finite cyclic groups:
\[
\xymatrix @-1pc{
\TT\ar[rr]\ar[dd] && \widetilde{\ccD\op}\ar[dd]\\ \\
\widetilde{\ccD}\ar[rr] && {\sf{PO}}\ar[dr]\\
&&& \Atl.
}
\]

In Section \ref{annularobjects}, we define the notion of an annular object in a category $\CC$. As $\ccD\op$ lives inside $\aD$ (in two ways), we will have notions of Hochschild and cyclic homology of annular objects in abelian categories. We define these notions and give some easy results in \ref{homology}.

\comment{
Finally, we make the previous subsection of lore rigorous in the Appendix. In \ref{temperleylieb}, we discuss the involutive category $\TL$ We find an isomorphic involutive subcategory $\aTL$ of $\Atl$, and we find the corresponding involutive subcategory $\tlD$ of $\aD$. In \ref{simplicial}, we find isomorphic copies of $\sD\op$ and $\ssD\op$ in $\Atl$ and $\TL$, which we denote $\sAtl$ and $\ssAtl$, respectively $\sTL$ and $\ssTL$.
}

\subsection{Acknowledgements}
The author would like to acknowledge and thank Vaughan Jones for his guidance and advice, Vijay Kodiyalam and V. S. Sunder for discussing the problem at length and for their hospitality at IMSc, Ved Gupta for proofreading and correcting an error in the first draft, and Emily Peters for her support and her help on drawing planar tangles (in fact, all tangles shown are adapted from \cite{0902.1294}). The author was partially supported by NSF grant DMS 0401734.

\section{The Category $\Atl$}\label{Atl}

\begin{nota}
All categories will be denoted by capital letters in the following sans-serif font: ${\sf{ABC}}$... The categories we discuss will be small, and we will write $X\in{\sf{A}}$ to denote that $X\in\Ob({\sf{A}})$, the set of objects of $\AAA$. We will write ${\sf{A}}(X,Y)$ to denote the set of morphisms $\varphi\colon X\to Y$ where $X,Y\in{\sf{A}}$, and we will write $\Mor({\sf{A}})$ to denote the collection of all morphisms in ${\sf{A}}$. In the sequel, objects of our categories will be the symbols $[n]$ for $n\in\Z_{\geq 0}\cup\{0\pm,\pm\}$. For simplicity and aesthetics, we will write ${\sf{A}}(m,n)$ instead of ${\sf{A}}([m],[n])$.
\end{nota}

\begin{defn}
A category $\AAA$ is called involutive if for all $X,Y\in\AAA$, there is a map $*\colon \AAA(X,Y)\to \AAA(Y,X)$ called the involution such that
\be
\item[(1)] $\id_X^*=\id_X$ for all $X\in\AAA$,
\item[(2)] $(T^*)^*=T$ for all $T\in \AAA(X,Y)$, and
\item[(3)] for all $X,Y,Z\in\AAA$ and all $T\in\AAA(X,Y)$ and $S\in \AAA(Y,Z)$, $(S\circ T)^*=T^*\circ S^*$.
\ee
In other words, there is a contravariant functor $*\colon \AAA\to \AAA$ of period two which fixes all objects. 
\end{defn}

\begin{defn} Suppose $\AAA$ and $\BB$ are categories and $F\colon \AAA\to \BB$ is a functor.
\be
\item[(1)] $F$ is called an isomorphism of categories if there is a functor $G\colon \BB\to \AAA$ such that $F\circ G=\id_{\BB}$ and $G\circ F=\id_{\AAA}$, the identity functors. In this case, we say categories $\AAA$ and $\BB$ are isomorphic, denoted $\AAA\cong\BB$.
\item[(2)] If $\AAA$ and $\BB$ are involutive, we say $F$ is involutive if it preserves the involution, i.e. $F(\varphi^*)=\varphi^*$ for all $\varphi\in\AAA(X,Y)$ for all $X,Y\in\AAA$.
\item[(3)] An isomorphism of involutive categories is an involutive isomorphism of said categories. 
\ee 
\end{defn}
\begin{rem}
It is clear that if $\AAA$ is involutive, then $\AAA\cong\AAA\op$.
\end{rem}
\subsection{Annular Tangles}\label{annulartangles}
We provide a definition of an annular $(m,n)$-tangle which is a fusion of the ideas in \cite{math/9909027} and \cite{MR2047470}.

\begin{defn}
An annular $(m,n)$-pretangle for $m,n\in\Z_{\geq 0}$ consists of the following data:
\item[(1)] The closed unit disk $D$ in $\C$,
\item[(2)] The skeleton of $T$, denoted $S(T)$, consisting of:
\be
\item[(a)] the boundary of $D$, denoted $D_0(T)$,
\item[(b)] the closed disk $D_1$  of radius $1/4$ in $\C$, whose boundary is denoted $D_1(T)$,
\item[(c)]  $2m$, respectively $2n$, distinct marked points on $D_1(T)$, respectively $D_0(T)$, called the boundary points of $D_i(T)$ for $i=0,1$. Usually we will call the boundary points of $D_0(T)$ external boundary points of $T$ and the boundary points of $D_1(T)$ internal boundary points.
\item[(d)] inside $D$, but outside $D_1$, there is a finite set of disjointly smoothly embedded curves called strings which are either closed curves, called loops, or whose boundaries are marked points of the $D_i(T)$'s and the strings meet each $D_i(T)$ transversally, $i=0,1$. Each marked point on $D_i(T)$, $i=0,1$ meets exactly one string. 
\ee
\item[(3)] The connected components of $D\setminus S(T)$ are called the regions of $T$ and are either shaded or unshaded so that regions whose closures meet have different shadings.
\itt{Definition}  An interval of $D_i(T)$, $i=0,1$, is a connected arc on $D_i(T)$ between two boundary points of $D_i(T)$. A simple interval of $D_i(T)$, $i=0,1$, is an interval of $D_i(T)$ in $T$ which touches only two (adjacent) boundary points. 
\item[(4)] For each $D_i(T)$, $i=0,1$, there is a distinguished simple interval of $D_i(T)$ denoted $*_i(T)$ whose interior meets an unshaded region.
\end{defn}

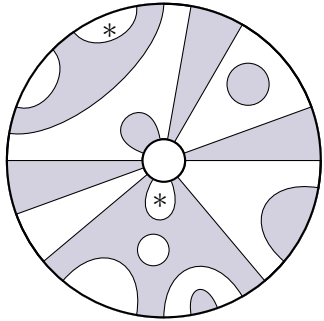
\begin{figure}[!ht]
\begin{tikzpicture}[PAdefn]
	\clip (0,0) circle (3cm);
	\draw[ultra thick] (0,0) circle (3cm);	
	\node at (0,0)  [empty box] (T) {};
	\filldraw[shaded] (0:4cm)--(0,0)--(20:4cm);	
	\filldraw[shaded] (60:4cm)--(0,0)--(80:4cm);
	\filldraw[shaded] (180:4cm)--(0,0)--(200:4cm);
	\filldraw[shaded] (220:5cm)--(0,0)--(310:5cm);		
	\filldraw[shaded] (T.100) .. controls ++(100:11mm) and ++           (160:11mm) .. (T.160);
	\filldraw[unshaded] (T.240) .. controls ++(240:11mm) and ++           (290:11mm) .. (T.290);
	\filldraw[shaded] (90:4cm)--(90:3cm) .. controls ++(270:11mm) and ++(350:11mm) .. (170:3cm) -- (170:4cm);
	\filldraw[unshaded] (100:4cm)--(100:3cm)  .. controls ++(280:6mm) and ++(305:6mm) ..         (125:3cm) -- (125:4cm);
	\filldraw[unshaded] (135:4cm)--(135:3cm)  .. controls ++(315:6mm) and ++(340:6mm) ..         (160:3cm) -- (160:4cm);
	\filldraw[unshaded] (270:4cm)--(270:3cm)  .. controls ++(90:11mm) and ++(120:11mm) ..         (300:3cm) -- (300:4cm);
	\filldraw[shaded] (280:4cm)--(280:3cm)  .. controls ++(100:6mm) and ++(110:6mm) .. (290:3cm) -- (290:4cm);
	\filldraw[unshaded] (230:4cm)--(230:3cm)  .. controls ++(50:11mm) and ++(80:11mm) ..         (260:3cm) -- (260:4cm);
	\filldraw[shaded] (320:4cm)--(320:3cm)  .. controls ++(140:11mm) and ++(170:11mm) ..         (350:3cm) -- (350:4cm);	
	\filldraw[shaded] (1.6,1.45) circle (.4cm);
	\filldraw[unshaded] (-.2,-1.7) circle (.3cm);
	\draw[ultra thick] (0,0) circle (3cm);	
	\node at (0,0)  [empty box] (T) {};
	\node at (T.260) [below] {$*$};
	\node at (110:3cm)[below] {$*$};
\end{tikzpicture}
\caption{An example of an annular tangle}
\end{figure}

\begin{rems}
\item[(1)]
If $m=0$, there are two kinds of annular $(0,n)$-pretangles depending on whether the region meeting $D_1(T)$ is unshaded or shaded. If the region meeting $D_1(T)$ is unshaded, we call $T$ an annular $(0+,n)$-pretangle, and if the region is shaded, we call $T$ an annular $(0-,n)$-pretangle. Likewise, when $n=0$, there are two kinds of annular $(m,0)$-pretangles. If the region meeting $D_0(T)$ is unshaded, we call $T$ an annular $(m,0+)$-pretangle, and if the region is shaded, we call $T$ an annular $(m,0-)$-pretangle. Additionally, we have annular $(0\pm,0\pm)$-pretangles and annular $(0\pm,0\mp)$-pretangles.
\item[(2)]
Loops may be shaded or unshaded.
\item[(3)] Starting at $*_i(T)$ on $D_i(T)$, we order the marked points of $D_i(T)$ clockwise. 
\end{rems}

\begin{defn}
An annular $(m,n)$-tangle is an orientation-preserving diffeomorphism class of an annular $(m,n)$-pretangle for $m,n\in\N\cup\{0\pm\}$. The diffeomorphisms preserve (but do not necessarily fix!) $D_0$ and $D_1$.
\end{defn}

\begin{defn}
Given an annular $(m,n)$-tangle $T$, and an annular $(l,m)$-tangle $S$, we define the annular $(l,m)$-tangle $T\circ S$ by isotoping $S$ so that $D_0(S)$, the marked points of $D_0(S)$, and $*_0(S)$, coincide with $D_1(T)$, the marked points of $D_1(T)$, and $*_1(T)$ respectively. The strings may then be joined at $D_1(T)$ and smoothed, and $D_1(T)$ is removed to obtain $T\circ S$ whose diffeomorphism class only depends on those of $T$ and $S$. 
\end{defn}

\begin{rem}
Note that in the case where $m=0-$ in the above defintion, there are no $*_0(S)$ and $*_1(T)$, but this information is not necessary to define the composite $T\circ S$.
\end{rem}

\begin{figure}[!ht]
$$
\begin{tikzpicture}[annular]
	\clip (0,0) circle (2cm);
	\draw[ultra thick] (0,0) circle (2cm);	
	\node at (0,0)  [empty box] (T) {};
	\node at (180:1.80cm) [right] {$*$};

	\filldraw[shaded] (90:4cm)--(0,0)--(270:4cm) -- (0:4cm);	
	\filldraw[unshaded] (T.-40) .. controls ++(-40:11mm) and ++           (40:11mm) .. (T.40);
	\filldraw[shaded] (T.140) .. controls ++(140:11mm) and ++           (220:11mm) .. (T.220);

	\draw[ultra thick] (0,0) circle (2cm);
	\node at (0,0)  [empty box] (T) {};
	\node at (T.120) [above] {$*$};		
\end{tikzpicture}
\circ
\begin{tikzpicture}[annular]
	\clip (0,0) circle (2cm);
	\draw[ultra thick] (0,0) circle (2cm);	
	\node at (0,0)  [empty box] (T) {};
	\node at (120:1.80cm) [below] {$*$};

	\filldraw[shaded] (90:4cm)--(0,0)--(270:4cm) -- (0:4cm);	
	\filldraw[unshaded] (-20:2cm) .. controls ++(160:7mm) and ++(200:7mm) .. (20:2cm) -- (20:3cm);
	\filldraw[shaded] (160:2cm) .. controls ++(-20:7mm) and ++(20:7mm) ..  (200:2cm) -- (200:3cm);

	\draw[ultra thick] (0,0) circle (2cm);
	\node at (0,0)  [empty box] (T) {};
	\node at (T.180) [left] {$*$};		
\end{tikzpicture}
=
\begin{tikzpicture}[annular]
	\clip (0,0) circle (2cm);
	\draw[ultra thick] (0,0) circle (2cm);	
	\node at (0,0)  [empty box] (T) {};
	\node at (180:1.40cm) [left] {$*$};

	\filldraw[shaded] (90:4cm)--(0,0)--(270:4cm) -- (0:4cm);	
	\filldraw[unshaded] (0:1cm) circle (.4cm);
	\filldraw[shaded] (180:1cm) circle (.4cm);
	\draw[ultra thick] (0,0) circle (2cm);
	\node at (0,0)  [empty box] (T) {};
	\node at (T.135) [above] {$*$};		
\end{tikzpicture}
$$
\caption{An example of composition of annular tangles}
\end{figure}
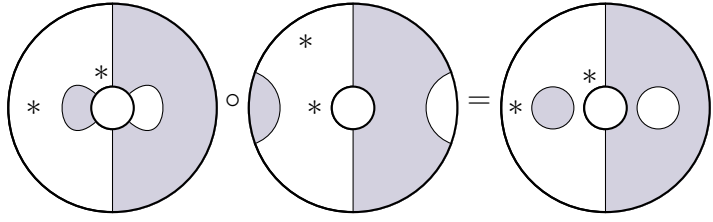

\begin{defn}
If $T$ is an annular $(m,n)$-tangle, we define $T^*$ to be the annular $(n,m)$-tangle obtained by reflecting $T$ about the circle of radius $3/4$, which switches $D_i(T)$ and $*_i(T)$, $i=0,1$. Clearly $(T^*)^*=T$ and $(T\circ S)^*=S^*\circ T^*$ for composable $S$ and $T$.
\end{defn}

\begin{figure}[!ht]
$$
\begin{tikzpicture}[annular]
	\clip (0,0) circle (2cm);
	\draw[ultra thick] (0,0) circle (2cm);	
	\node at (0,0)  [empty box] (T) {};
	\node at (-170:2cm) [right] {$*$};

	\filldraw[shaded] (0:2cm) .. controls ++(180:7mm) and ++(-135:7mm) .. (45:2cm) -- (45:3cm)--(90:3cm) -- (90:2cm) --(0,0) -- (-45:2cm) -- (-45:3cm);		
	\filldraw[shaded] (-90:2cm) .. controls ++(90:7mm) and ++(45:7mm) .. (-135:2cm) -- (-135:3cm);
	\filldraw[shaded] (-180:2cm) .. controls ++(0:7mm) and ++(-45:7mm) .. (135:2cm) -- (135:3cm);
	\filldraw[shaded] (T.180) .. controls ++(180:11mm) and ++(-120:11mm) .. (T.-120);

	\draw[ultra thick] (0,0) circle (2cm);	
	\node at (0,0)  [empty box] (T) {};
	\node at (T.160) [above] {$*$};
\end{tikzpicture}^*=
\begin{tikzpicture}[annular]
	\clip (0,0) circle (2cm);
	\draw[ultra thick] (0,0) circle (2cm);	
	\node at (0,0)  [empty box] (T) {};
	\node at (160:2cm) [right] {$*$};

	\filldraw[shaded] (0,0) -- (90:3cm) -- (0:4cm) -- (-45:3cm) -- (0,0) .. controls ++(-15:2cm) and ++(60:2cm) .. (0,0);	
	\filldraw[shaded] (180:2cm) .. controls ++(0:7mm) and ++(45:7mm) .. (-135:2cm) -- (-135:3cm);
	\filldraw[shaded] (T.-75) .. controls ++(-75:11mm) and ++(-135:11mm) .. (T.-135);
	\filldraw[shaded] (T.180) .. controls ++(180:11mm) and ++(120:11mm) .. (T.120);		

	\draw[ultra thick] (0,0) circle (2cm);	
	\node at (0,0)  [empty box] (T) {};
	\node at (T.-150) [left] {$*$};
\end{tikzpicture}
$$
\caption{An example of the adjoint of an annular tangle}
\end{figure}
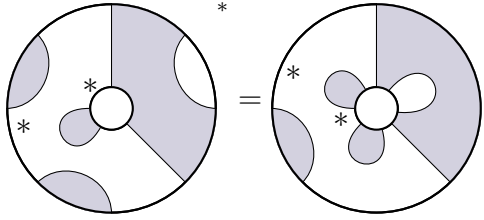

\begin{defn}
Let $T$ be an annular $(m,n)$-tangle.
\be
\itt{Caps} A cap of $T$ is a string that connects two internal boundary points. The set of caps of $T$ will be denoted $\caps(T)$. 
\itm{\partial \Lambda} If $\Lambda\in\caps(T)$, there is a unique interval of $D_1(T)$, denoted $\partial \Lambda$, such that $\Lambda\cup\partial \Lambda$ is a closed loop (which is not smooth at two points) which does not contain $D_1$ in its interior. Using $\partial \Lambda$, the cap $\Lambda$ inherits an orientation as $D_1(T)$ is oriented clockwise. Denote this orientation by an arrow on $\Lambda$. 
\itt{Index} We define the cap index of $\Lambda$, denoted $\ind(\Lambda)$, to be the number of the marked point to which the arrow points. The set of cap indices of $T$ forms an increasing sequence, which we denote $\capind(T)$.
\itm{B(\Lambda)} For $\Lambda\in\caps(T)$, we let $B(\Lambda)=\set{\Lambda'\in\caps(T)}{\partial \Lambda'\subseteq \partial \Lambda}$, and we say an element $\Lambda'\in B(\Lambda)$ is bounded by $\Lambda$ or that $\Lambda$ bounds $\Lambda'$.
\ee
\end{defn}

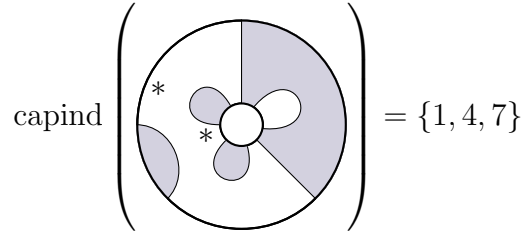
\begin{figure}[!ht]
$$
\capind\left(
\begin{tikzpicture}[annular]
	\clip (0,0) circle (2cm);
	\draw[ultra thick] (0,0) circle (2cm);	
	\node at (0,0)  [empty box] (T) {};
	\node at (160:2cm) [right] {$*$};

	\filldraw[shaded] (0,0) -- (90:3cm) -- (0:4cm) -- (-45:3cm) -- (0,0) .. controls ++(-15:2cm) and ++(60:2cm) .. (0,0);	
	\filldraw[shaded] (180:2cm) .. controls ++(0:7mm) and ++(45:7mm) .. (-135:2cm) -- (-135:3cm);
	\filldraw[shaded] (T.-75) .. controls ++(-75:11mm) and ++(-135:11mm) .. (T.-135);
	\filldraw[shaded] (T.180) .. controls ++(180:11mm) and ++(120:11mm) .. (T.120);		

	\draw[ultra thick] (0,0) circle (2cm);	
	\node at (0,0)  [empty box] (T) {};
	\node at (T.-150) [left] {$*$};
\end{tikzpicture}
\right) = \{1,4,7 \}
$$
\caption{An example of cap indices}
\end{figure}

\begin{defn}
Let $T$ be an annular $(m,n)$-tangle.
\be 
\itt{Cups} A cup $V$ of $T$ is a string that connects two external boundary points. The set of cups of $T$ will be denoted $\cups(T)$. 
\itm{\partial V} If $V\in\cups(T)$, there is a unique interval of $D_0(T)$, denoted $\partial V$, such that $V\cup\partial V$ is a closed loop (which is not smooth at two points) which does not contain $D_1$ in its interior. Using $\partial V$, the cup $V$ inherits an orientation as $D_0(T)$ is oriented clockwise. Denote this orientation by an arrow on $V$. 
\itt{Index} We define the cup index of $V$, denoted $\ind(V)$, to be the number of the marked point to which the arrow points. The set of cup indices of $T$ forms an increasing sequence, which we denote $\cupind(T)$.
\itm{B(V)} For $V\in\cups(T)$, we let $B(V)=\set{V'\in\cups(T)}{\partial V'\subseteq \partial V}$, and we say an element $V'\in B(V)$ is bounded by $V$ or that $V$ bounds $V'$.
\ee
\end{defn}

\begin{rem}
Note $\capind(T)=\cupind(T^*)$ for all annular tangles $T$.
\end{rem}

\begin{defn} Suppose $T$ is an annular $(m,n)$-tangle.
\be
\itm{\ts(T)}
A through string is a string of $T$ which connects an internal boundary point of $T$ to an external boundary point of $T$. The set of through strings is denoted $\ts(T)$. Note that $|\ts(T)|\in 2\Z_{\geq 0}$. We order $\ts(T)$ clockwise starting at $*_0(T)$, so each through string of $T$ has a number.
\itm{\ts_0(T)} Suppose $T$ has a through string. Using $*_0(T)$ as our reference, we go counterclockwise along $D_0(T)$ to the first through string, which is denoted $\ts_0(T)$. Note the number of $\ts_0(T)$ is $|\ts(T)|$.
\itm{\ts_1(T)} Suppose $T$ has a through string. Using $*_1(T)$ as our reference, we go counterclockwise along $D_1(T)$ to the first through string, which is denoted $\ts_1(T)$. We denote the number of $\ts_1(T)$ by $\#\ts_1(T)$.
\itm{\rel_*(T)} We define the relative star position of $T$, denoted $\rel_*(T)$, as follows:
\be
\item[(1)] Suppose $T$ has an odd number $k$ of non-contractible loops. Then there is a unique region $R$ which touches both a non-contractible loop and $D_1(T)$. If $R$ is unshaded, we define $\rel_*(T)$ to be the symbol $\pm(k)$, and if $R$ is shaded, we define $\rel_*(T)$ to be the symbol $\mp(k)$. This notation signifies the shading switches from unshaded to shaded, respectively shaded to unshaded, as we read $T$ from $D_1(T)$ to $D_0(T)$.
\item[(2)] Suppose $T$ has an even number $k$ of non-contractible loops. If $k=0$, then there is a unique region $R$ which touches both $D_0(T)$ and $D_1(T)$. If $k\geq 1$, then there is a unique region $R$ which touches both a non-contractible loop and $D_1(T)$. If $R$ is unshaded, we define $\rel_*(T)$ to be the symbol $+(k)$, and if $R$ is shaded, we define $\rel_*(T)$ to be the symbol $-(k)$. 
\item[(3)] Suppose $T$ has a through string. We define 
$$
\rel_*(T)=\left\lfloor \frac{\#\ts_1(T)}{2}\right\rfloor \mod\left( \frac{|\ts(T)|}{2}\right)\in \left\{ 0,1,\dots, \frac{|\ts(T)|}{2}-1\right\}.
$$
\ee 
\ee
\end{defn}

\begin{figure}[!ht]
$$
\rel_*\left(
\begin{tikzpicture}[PAdefn]
	\clip (0,0) circle (3cm);
	\draw[ultra thick] (0,0) circle (3cm);	
	\node at (0,0)  [empty box] (T) {};

	\filldraw[shaded] (0:4cm)--(0,0)--(20:4cm);	
	\filldraw[shaded] (60:4cm)--(0,0)--(80:4cm);
	\filldraw[shaded] (180:4cm)--(0,0)--(200:4cm);
	\filldraw[shaded] (220:5cm)--(0,0)--(310:5cm);		

	\filldraw[shaded] (T.100) .. controls ++(100:11mm) and ++           (160:11mm) .. (T.160);
	\filldraw[unshaded] (T.240) .. controls ++(240:11mm) and ++           (290:11mm) .. (T.290);
	\filldraw[shaded] (90:4cm)--(90:3cm) .. controls ++(270:11mm) and ++(350:11mm) .. (170:3cm) -- (170:4cm);
	\filldraw[unshaded] (100:4cm)--(100:3cm)  .. controls ++(280:6mm) and ++(305:6mm) ..         (125:3cm) -- (125:4cm);
	\filldraw[unshaded] (135:4cm)--(135:3cm)  .. controls ++(315:6mm) and ++(340:6mm) ..         (160:3cm) -- (160:4cm);
	\filldraw[unshaded] (270:4cm)--(270:3cm)  .. controls ++(90:11mm) and ++(120:11mm) ..         (300:3cm) -- (300:4cm);
	\filldraw[shaded] (280:4cm)--(280:3cm)  .. controls ++(100:6mm) and ++(110:6mm) .. (290:3cm) -- (290:4cm);
	\filldraw[unshaded] (230:4cm)--(230:3cm)  .. controls ++(50:11mm) and ++(80:11mm) ..         (260:3cm) -- (260:4cm);
	\filldraw[shaded] (320:4cm)--(320:3cm)  .. controls ++(140:11mm) and ++(170:11mm) ..         (350:3cm) -- (350:4cm);			

	\draw[ultra thick] (0,0) circle (3cm);	
	\node at (0,0)  [empty box] (T) {};
	\node at (T.255) [below] {$*$};
	\node at (110:3cm)[below] {$*$};
\end{tikzpicture}
\right)=2
$$
\caption{An example of relative star position}
\end{figure}
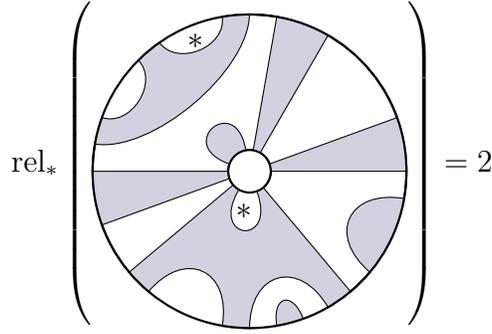

\subsection{``Generators and Relations" of $\Atl$}

\begin{defn}
Suppose $T$ is an annular tangle. A loop of $T$ is called contractible if it is contractible in 
$D\setminus D_1$. Otherwise it is called non-contractible.
\end{defn}

\begin{defn}[Atl Tangle]
An annular $(m,n)$-tangle $T$ is called an Atl $(m,n)$-tangle if $T$ has no contractible loops.
\end{defn}

\begin{defn}
Let $\Atl$ be the following small category:
\itt{Objects} $[n]$ for $n\in\N\cup\{0\pm\}$
\itt{Morphisms} Given $m,n\in \N\cup \{0\pm\}$, $\Atl(m,n)$ is the set of all triples $(T,c_+,c_-)$ where $T$ is an Atl $(m,n)$-tangle and $c_+,c_-\geq 0$.
\itt{Composition} Given $(S,a_+,a_-)\in \Atl(m,n)$ and $(T,b_+,b_-)\in \Atl(l,m)$, we define $(S,a_+,a_-)\circ (T,b_+,b_-)\in \Atl(l,n)$ to be the triple $(R,c_+,c_-)$ obtained as follows: let $R_0$ be the annular $(l,n)$-tangle $S\circ T$. Let $d_+$, respectively $d_-$, be the number of shaded, respectively unshaded, contractible loops. Let $R$ be the Atl $(l,n)$-tangle obtained from $R_0$ by removing all contractible loops, and set $c_\pm=a_\pm+b_\pm+d_\pm$.
\end{defn}

\begin{rem}
For simplicity and aesthetics, we write $T$ for the morphism $(T,0,0)\in \Mor(\Atl)$.
\end{rem}

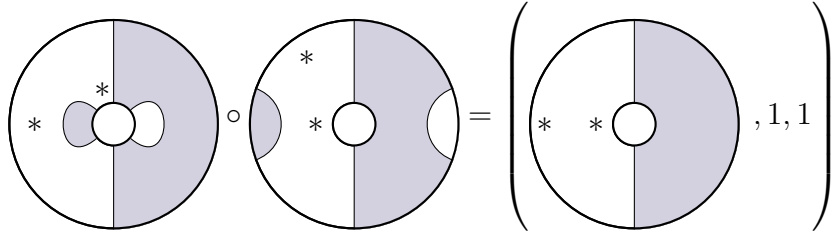
\begin{figure}[!ht]
$$
\begin{tikzpicture}[annular]
	\clip (0,0) circle (2cm);
	\draw[ultra thick] (0,0) circle (2cm);	
	\node at (0,0)  [empty box] (T) {};
	\node at (180:1.80cm) [right] {$*$};

	\filldraw[shaded] (90:4cm)--(0,0)--(270:4cm) -- (0:4cm);	
	\filldraw[unshaded] (T.-40) .. controls ++(-40:11mm) and ++           (40:11mm) .. (T.40);
	\filldraw[shaded] (T.140) .. controls ++(140:11mm) and ++           (220:11mm) .. (T.220);

	\draw[ultra thick] (0,0) circle (2cm);
	\node at (0,0)  [empty box] (T) {};
	\node at (T.120) [above] {$*$};		
\end{tikzpicture}
\circ
\begin{tikzpicture}[annular]
	\clip (0,0) circle (2cm);
	\draw[ultra thick] (0,0) circle (2cm);	
	\node at (0,0)  [empty box] (T) {};
	\node at (120:1.80cm) [below] {$*$};

	\filldraw[shaded] (90:4cm)--(0,0)--(270:4cm) -- (0:4cm);	
	\filldraw[unshaded] (-20:2cm) .. controls ++(160:7mm) and ++(200:7mm) .. (20:2cm) -- (20:3cm);
	\filldraw[shaded] (160:2cm) .. controls ++(-20:7mm) and ++(20:7mm) ..  (200:2cm) -- (200:3cm);

	\draw[ultra thick] (0,0) circle (2cm);
	\node at (0,0)  [empty box] (T) {};
	\node at (T.180) [left] {$*$};		
\end{tikzpicture}
=\left(
\begin{tikzpicture}[annular]
	\clip (0,0) circle (2cm);
	\draw[ultra thick] (0,0) circle (2cm);	
	\node at (0,0)  [empty box] (T) {};
	\node at (180:1.40cm) [left] {$*$};

	\filldraw[shaded] (90:4cm)--(0,0)--(270:4cm) -- (0:4cm);	

	\draw[ultra thick] (0,0) circle (2cm);
	\node at (0,0)  [empty box] (T) {};
	\node at (T.180) [left] {$*$};		
\end{tikzpicture}\hs,
1,1\right)
$$
\caption{An example of composition in $\Atl$}
\end{figure}

\begin{defn} We give the following names to the following distinguished Atl $(n,m)$-tangles:
\item[(A)] Let $a_1$ be the only Atl $(1,0+)$-tangle with no loops, and let $a_2$ be the only Atl $(1,0-)$-tangle with no loops. 
\begin{figure}[!ht]
\begin{tikzpicture}[annular]
	\clip (0,0) circle (2cm);
	\draw[ultra thick] (0,0) circle (2cm);	
	\node at (0,0)  [empty box] (T) {};
	\node at (95:1.60cm) [right] {$*$};
	\filldraw[shaded] (T.-40) .. controls ++(-40:11mm) and ++           (40:11mm) .. (T.40);
	\draw[ultra thick] (0,0) circle (2cm);
	\node at (0,0)  [empty box] (T) {};
	\node at (T.70) [above] {$*$};		
\end{tikzpicture}$\hs,$
\begin{tikzpicture}[annular]
	\clip (0,0) circle (2cm);
	\draw[ultra thick][shaded] (0,0) circle (2cm);	
	\node at (0,0)  [empty box] (T) {};
	\filldraw[unshaded] (T.-40) .. controls ++(-40:11mm) and ++           (40:11mm) .. (T.40);
	\draw[ultra thick] (0,0) circle (2cm);
	\node at (0,0)  [empty box] (T) {};
	\node at (T.0) [right] {$*$};		
\end{tikzpicture}
\caption{$a_1\in\Atl(1,0+)$ and $a_2\in\Atl(1,0-)$}
\end{figure}
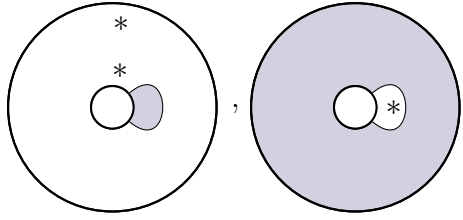
For $n\geq 2$ and $i\in\{1,\dots,2n\}$, let $a_i$ be the Atl $(n,n-1)$-tangle whose $i\thh$ and $(i+1)\thh$ (modulo $2n$) internal boundary point are joined by a string and all other internal boundary points are connected to external boundary points such that
\be
\item[(i)] If $i=1$, then the first external point is connected to the third internal point.
\item[(ii)] If $1<i<2n$, then the first external point is connected to the first internal point.
\item[(iii)] If $i=2n$, then the first external point is connected to the $(2n-1)\thh$ internal point.
\ee
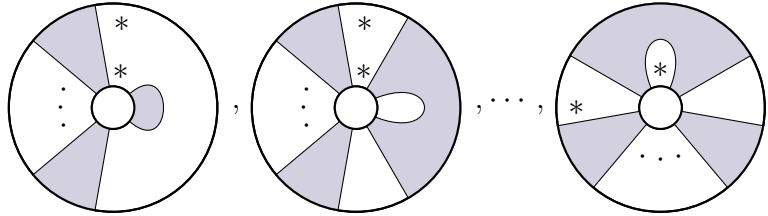
\begin{figure}[!ht]
\begin{tikzpicture}[annular]
	\clip (0,0) circle (2cm);
	\draw[ultra thick] (0,0) circle (2cm);	
	\node at (0,0)  [empty box] (T) {};
	\node at (95:1.60cm) [right] {$*$};
	\filldraw[shaded] (100:4cm)--(0,0)--(140:4cm);	
	\filldraw[shaded] (220:4cm)--(0,0)--(260:4cm);
		
	\filldraw[shaded] (T.-40) .. controls ++(-40:11mm) and ++           (40:11mm) .. (T.40);
	\draw[ultra thick] (0,0) circle (2cm);
	\node at (0,0)  [empty box] (T) {};
	\node at (180:1cm) {$\cdot$};
	\node at (160:1cm) {$\cdot$};
	\node at (200:1cm) {$\cdot$};
	\node at (T.70) [above] {$*$};		
\end{tikzpicture}$\hs,$
\begin{tikzpicture}[annular]
	\clip (0,0) circle (2cm);
	\draw[ultra thick] (0,0) circle (2cm);	
	\node at (0,0)  [empty box] (T) {};
	\node at (95:1.60cm) [right] {$*$};
	\filldraw[shaded] (100:4cm)--(0,0)--(140:4cm);	
	\filldraw[shaded] (220:4cm)--(0,0)--(260:4cm);
	\filldraw[shaded] (0,0) -- (60:3cm) -- (0:4cm) -- (-60:3cm) -- (0,0) .. controls ++(-30:2cm) and ++(30:2cm) .. (0,0);
	\draw[ultra thick] (0,0) circle (2cm);
	\node at (0,0)  [empty box] (T) {};
	\node at (180:1cm) {$\cdot$};
	\node at (160:1cm) {$\cdot$};
	\node at (200:1cm) {$\cdot$};
	\node at (T.70) [above] {$*$};		
\end{tikzpicture}$\hs,\cdots,$
\begin{tikzpicture}[annular]
	\clip (0,0) circle (2cm);
	\draw[ultra thick] (0,0) circle (2cm);	
	\node at (0,0)  [empty box] (T) {};
	\node at (180:1.90cm) [right] {$*$};
	\filldraw[shaded] (310:4cm)--(0,0)--(350:4cm);	
	\filldraw[shaded] (190:4cm)--(0,0)--(230:4cm);
	\filldraw[shaded] (0,0) -- (30:3cm) -- (90:4cm) -- (150:3cm) -- (0,0) .. controls ++(120:2cm) and ++(60:2cm) .. (0,0);
	\draw[ultra thick] (0,0) circle (2cm);
	\node at (0,0)  [empty box] (T) {};
	\node at (270:1cm) {$\cdot$};
	\node at (290:1cm) {$\cdot$};
	\node at (250:1cm) {$\cdot$};
	\node at (T.90) [above] {$*$};	
\end{tikzpicture}
\caption{$a_1,a_2, \cdots,a_{2n}\in\Atl(n,n-1)$. (without the dots, $n=3$)}
\end{figure}
\item[(B)]
 Let $b_1$ be the only Atl $(0+,1)$-tangle with no loops, and let $b_2$ be the only Atl $(0-,1)$-tangle with no loops.
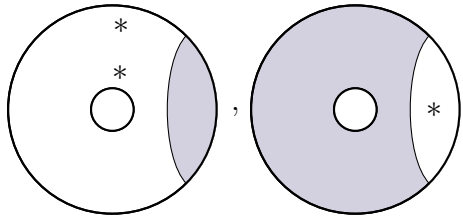
\begin{figure}[!ht]
\begin{tikzpicture}[annular]
	\clip (0,0) circle (2cm);
	\draw[ultra thick] (0,0) circle (2cm);	
	\node at (0,0)  [empty box] (T) {};
	\node at (95:1.60cm) [right] {$*$};
	\filldraw[shaded] (-45:3cm) --(-45:2cm) .. controls ++(135:7mm) and ++(-135:7mm) ..         (45:2cm) -- (45:3cm);	
	\draw[ultra thick] (0,0) circle (2cm);
	\node at (0,0)  [empty box] (T) {};
	\node at (T.70) [above] {$*$};		
\end{tikzpicture}$\hs,$
\begin{tikzpicture}[annular]
	\clip (0,0) circle (2cm);
	\draw[ultra thick][shaded] (0,0) circle (2cm);	
	\node at (0,0)  [empty box] (T) {};

	\filldraw[unshaded] (-45:3cm) --(-45:2cm) .. controls ++(135:7mm) and ++(-135:7mm) ..         (45:2cm) -- (45:3cm);	
	\node at (0:1.80cm) [left] {$*$};
	\draw[ultra thick] (0,0) circle (2cm);
	\node at (0,0)  [empty box] (T) {};
\end{tikzpicture}
\caption{$b_1\in\Atl(0+,1)$ and $b_2\in\Atl(0-,1)$}
\end{figure}
For $n\geq 1$ and $i\in\{1,\dots,2n+2\}$, let $b_i$ be the Atl $(n,n+1)$-tangle whose $i\thh$ and $(i+1)\thh$ (modulo $2n+2$) external boundary point are joined by a string and all other internal boundary points are connected to external boundary points such that
\be
\item[(i)] If $i=1$, then the third external point is connected to the first internal point.
\item[(ii)] If $1<i$, then the first external point is connected to the first internal point.
\item[(iii)] If $i=2n+2$, then the first internal point is connected to the $(2n+1)\thh$ external point.
\ee
\begin{figure}[!ht]
\begin{tikzpicture}[annular]
	\clip (0,0) circle (2cm);
	\draw[ultra thick] (0,0) circle (2cm);	
	\node at (0,0)  [empty box] (T) {};
	\node at (135:1.80cm) [right] {$*$};
	\filldraw[shaded] (-20:4cm)--(0,0)--(20:4cm);	
	\filldraw[shaded] (160:4cm)--(0,0)--(200:4cm);
	\filldraw[shaded] (250:4cm)--(0,0)--(290:4cm);		
	\filldraw[shaded] (70:2cm) .. controls ++(250:7mm) and ++(290:7mm) ..         (110:2cm) -- (110:3cm);
	\draw[ultra thick] (0,0) circle (2cm);
	\node at (0,0)  [empty box] (T) {};
	\node at (-135:1cm) {$\cdot$};
	\node at (-145:1cm) {$\cdot$};
	\node at (-125:1cm) {$\cdot$};
	\node at (T.90) [above] {$*$};		
\end{tikzpicture}$\hs,$
\begin{tikzpicture}[annular]
	\clip (0,0) circle (2cm);
	\draw[ultra thick] (0,0) circle (2cm);	
	\node at (0,0)  [empty box] (T) {};
	\node at (120:1.80cm) [right] {$*$};
	\filldraw[shaded] (160:4cm)--(0,0)--(200:4cm);
	\filldraw[shaded] (250:4cm)--(0,0)--(290:4cm);
	\filldraw[shaded] (25:2cm) .. controls ++(205:7mm) and ++(245:7mm) .. (65:2cm) -- (65:3cm)--(90:3cm) -- (90:2cm) --(0,0) -- (0:2cm) -- (0:3cm);
	\draw[ultra thick] (0,0) circle (2cm);
	\node at (0,0)  [empty box] (T) {};
	\node at (-135:1cm) {$\cdot$};
	\node at (-145:1cm) {$\cdot$};
	\node at (-125:1cm) {$\cdot$};
	\node at (T.120) [above] {$*$};		
\end{tikzpicture}$\hs,\cdots,$
\begin{tikzpicture}[annular]
	\clip (0,0) circle (2cm);
	\draw[ultra thick] (0,0) circle (2cm);	
	\node at (0,0)  [empty box] (T) {};
	\node at (135:1.90cm) [right] {$*$};
	\filldraw[shaded] (-20:4cm)--(0,0)--(20:4cm);
	\filldraw[shaded] (250:4cm)--(0,0)--(290:4cm);
	\filldraw[shaded] (115:2cm) .. controls ++(295:7mm) and ++(335:7mm) .. (155:2cm) -- (155:3cm)--(180:3cm) -- (180:2cm) --(0,0) -- (90:2cm) -- (90:3cm);
	\draw[ultra thick] (0,0) circle (2cm);
	\node at (0,0)  [empty box] (T) {};
	\node at (315:1cm) {$\cdot$};
	\node at (305:1cm) {$\cdot$};
	\node at (325:1cm) {$\cdot$};
	\node at (T.-155)[below] {$*$};	
\end{tikzpicture}
\caption{$b_1,b_2, \cdots, b_{2n+2}\in\Atl(n,n+1)$ (without the dots, $n=3$)}
\end{figure}
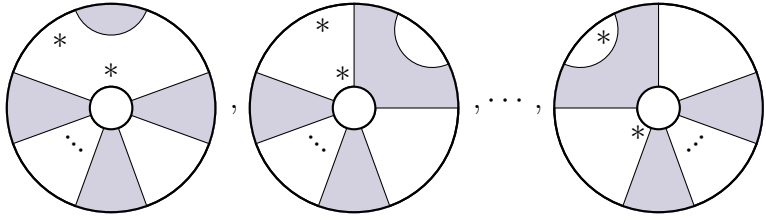
\item[(T)]
For $n=1$, let $t$ be the identity $(1,1)$-tangle. For $n\geq 2$, let $t$ be the Atl $(n,n)$-tangle where all internal points are connected to external point such that the third external point is connected to the first internal point. 
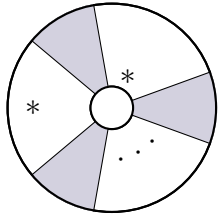
\begin{figure}[!ht]
\begin{tikzpicture}[annular]
	\clip (0,0) circle (2cm);
	\draw[ultra thick] (0,0) circle (2cm);	
	\node at (0,0)  [empty box] (T) {};
	\node at (180:1.80cm) [right] {$*$};
	\filldraw[shaded] (340:4cm)--(0,0)--(20:4cm);	
	\filldraw[shaded] (100:4cm)--(0,0)--(140:4cm);
	\filldraw[shaded] (220:4cm)--(0,0)--(260:4cm);		
	\draw[ultra thick] (0,0) circle (2cm);
	\node at (0,0)  [empty box] (T) {};
	\node at (300:1cm) {$\cdot$};
	\node at (320:1cm) {$\cdot$};
	\node at (280:1cm) {$\cdot$};
	\node at (T.45) [above] {$*$};		
\end{tikzpicture}
\caption{$t\in\Atl(n,n)$ (without the dots, $n=3$)}
\end{figure}
\end{defn}

\begin{thm}\label{Atlrelations}
The following relations hold in $\Atl$:
\item[(1)] $a_ia_j = a_{j-2}a_i$ for $i<j-1$ and $(i,j)\neq (1,2n)$,
\item[(2)] $b_ib_j=b_{j+2}b_i$ for $i\leq j$ and $(i,j)\neq (1,2n+2)$,
\item[(3)] $t^n=\id_{[n]}$,
\item[(4)] $a_it = ta_{i-2}$ for $i\geq 3$, 
\item[(5)] $b_i t = tb_{i-2}$ for $i\geq 3$,
\item[(6)] $(\id_{[0+]},1,0)=a_1b_1\in\Atl(0+,0+)$ and $(\id_{[0+]},0,1)=a_2b_2\in\Atl(0-,0-)$. If $a_ib_j\in\Atl(n, n)$ with $n\geq 1$, then 
$$\D
a_ib_j=
\begin{cases}
t^{-1} &\hsp{if} (i,j)=(1,2n+2)\\
b_{j-2}a_i &\hsp{if} i<j-1, (i,j)\neq (1,2n+2)\\
\id_{[n]} &\hsp{if} i=j-1\\
(\id_{[n]},1,0) &\hsp{if} i=j\hsp{and $i$ is odd} \\
(\id_{[n]},0,1) &\hsp{if} i=j\hsp{and $i$ is even} \\
\id_{[n]} &\hsp{if} i=j+1\\
b_{j}a_{i-2} &\hsp{if} i>j+1, (i,j)\neq (2n+2,1)\\
t &\hsp{if} (i,j)=(2n+2,1)
\end{cases}$$
\item[(7)] $(\id_{[n]},1,0)$ and $(\id_{[n]},0,1)$ commute with all $(T,c_+,c_-)\in\Atl(n,n)$ where $n\in\N\cup\{0\pm\}$.
\end{thm}
\begin{proof}
These relations can be easily verified by drawing pictures.
\end{proof}

\subsection{Involution and Tangle Type}

\begin{defn}
The map $*\colon \Atl\to\Atl$ given by $[n]^*=[n]$ for all $n\in\N\cup\{0\pm\}$ and $(T,c_+,c_-)^*=(T^*,c_+,c_-)$ defines an involution on $\Atl$.
\end{defn}
\begin{cor}
We have an isomorphism of categories $\Atl\cong\Atl\op$.
\end{cor}

\begin{prop}\label{Atlinvolution}
The involution on $\Atl$ satisfies
\item[(A/B)] 
$a_i^*=b_i$ for $i=1,2$ if $a_1\in \Atl(1,0+)$ and $a_2\in\Atl(1,0-)$. For $n\geq 2$ and $a_i\in\Atl(n,n-1)$, so $i\in\{1,\dots,2n\}$, $a_i^*=b_i\in\Atl(n-1,n)$.
\item[(T)] For $n\in\N$ and $t\in\Atl(n,n)$, $t^*=t^{-1}$.
\item[(D)] For $n\in\N\cup\{0\pm\}$, $(\id_{[n]},1,0)^*=(\id_{[n]},1,0)$ and $(\id_{[n]},0,1)^*=(\id_{[n]},0,1)$.
\end{prop}
\begin{proof}
Obvious.
\end{proof}

\begin{defn}\label{types}
An Atl $(m,n)$-tangle $T$ is said to be of
\be
\itt{Type I} if $T$ is either $\id_{[n]}$ for some $n\in\N\cup\{\pm 0\}$, or $T$ has no cups, at least one cap, and no non-contractible loops, with the limitation on $*_0(T)$ that exactly one of the following occurs:
\be
\item[(I-1)] There are no through strings, so $*_0(T)$ is uniquely determined. Note that if $n=0-$, then there is no $*_0(T)$.
\item[(I-2)] There are through strings. Using $*_1(T)$ as our reference, we go counterclockwise to the first through string, and travel outward until we reach a marked point $p$ of $D_0(T)$. The simple interval meeting $p$ whose interior touches an unshaded region is $*_0(T)$.
\ee
\itt{Type II} if $T$ has no cups or caps, so $T$ is a power of the rotation (including the identity tangle) or an annular $(0,0)$-tangle with $k$ non-contractible loops (here we do not specify $0\pm$).
\itt{Type III} if $T$ is either $\id_{[n]}$ for some $n\in\N\cup\{\pm 0\}$, or $T$ has no caps, at least one cup, and no non-contractible loops, with the limitation on  $*_1(T)$, that exactly one of the following occurs: 
\be
\item[(III-1)] There are no through strings, so $*_1(T)$ is uniquely determined. Note that if $m=0-$, then there is no $*_1(T)$.
\item[(III-2)] There are through strings. Using $*_0(T)$ as our reference, we go counterclockwise to the first through string, and travel outward until we reach a marked point $p$ of $D_1(T)$. The simple interval meeting $p$ whose interior touches an unshaded region is $*_1(T)$.
\ee
\ee
Denote the set of all tangles of Type i by $\T_i$, and denote the set of all $(m,n)$-tangles of Type $i$  by $\T_i(m,n)$  for $i\in\{I,II,III\}$.
\end{defn}

\begin{rem}\label{typerem}
Note that
\be
\item[(1)] the $a_i$'s are all Type I, and
\item[(2)] the $b_i$'s are all Type III.
\ee
\end{rem}

\begin{figure}[!ht]
$$
\begin{tikzpicture}[annular]
	\clip (0,0) circle (2cm);
	\draw[ultra thick] (0,0) circle (2cm);	
	\node at (0,0)  [empty box] (T) {};
	\filldraw[shaded] (0:4cm)--(0,0)--(20:4cm);	
	\filldraw[shaded] (60:4cm)--(0,0)--(80:4cm);
	\filldraw[shaded] (180:4cm)--(0,0)--(200:4cm);
	\filldraw[shaded] (220:5cm)--(0,0)--(310:5cm);		
	\filldraw[shaded] (T.100) .. controls ++(100:11mm) and ++           (160:11mm) .. (T.160);
	\filldraw[unshaded] (T.240) .. controls ++(240:11mm) and ++           (290:11mm) .. (T.290);
	\draw[ultra thick] (0,0) circle (2cm);	
	\node at (0,0)  [empty box] (T) {};
	\node at (T.260) [below] {$*$};
	\node at (-20:2cm)[left] {$*$};
\end{tikzpicture}\hs,\hs
\begin{tikzpicture}[annular]
	\clip (0,0) circle (2cm);
	\draw[ultra thick] (0,0) circle (2cm);	
	\node at (0,0)  [empty box] (T) {};
	\filldraw[shaded] (0:4cm)--(0,0)--(20:4cm);	
	\filldraw[shaded] (60:4cm)--(0,0)--(80:4cm);
	\filldraw[shaded] (180:4cm)--(0,0)--(200:4cm);
	\filldraw[shaded] (220:5cm)--(0,0)--(310:5cm);		
	\filldraw[shaded] (90:4cm)--(90:2cm) .. controls ++(270:11mm) and ++(350:11mm) .. (170:2cm) -- (170:4cm);
	\filldraw[unshaded] (100:4cm)--(100:2cm)  .. controls ++(280:6mm) and ++(305:6mm) ..         (125:2cm) -- (125:4cm);
	\filldraw[unshaded] (135:4cm)--(135:2cm)  .. controls ++(315:6mm) and ++(340:6mm) ..         (160:2cm) -- (160:4cm);
	\filldraw[unshaded] (270:4cm)--(270:2cm)  .. controls ++(90:11mm) and ++(120:11mm) ..         (300:2cm) -- (300:4cm);
	\filldraw[shaded] (280:4cm)--(280:2cm)  .. controls ++(100:6mm) and ++(110:6mm) .. (290:2cm) -- (290:4cm);
	\filldraw[unshaded] (230:4cm)--(230:2cm)  .. controls ++(50:11mm) and ++(80:11mm) ..         (260:2cm) -- (260:4cm);
	\filldraw[shaded] (320:4cm)--(320:2cm)  .. controls ++(140:11mm) and ++(170:11mm) ..         (350:2cm) -- (350:4cm);			
	\draw[ultra thick] (0,0) circle (2cm);	
	\node at (0,0)  [empty box] (T) {};
	\node at (T.110) [above] {$*$};
	\node at (110:2cm)[below] {$*$};
\end{tikzpicture}
$$
\caption{Examples of tangles of Types I and III}
\end{figure}
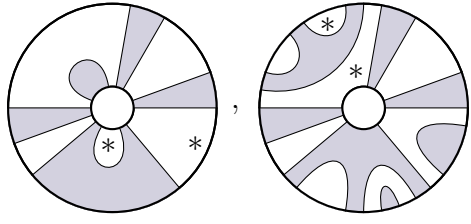

\begin{nota}
We will use the notation $s_+=a_2b_1\in\Atl(0+,0-)$ and $s_-=a_1b_2\in\Atl(0-,0+)$.
\end{nota}

\begin{figure}[!ht]
$$
\begin{tikzpicture}[annular]
	\clip (0,0) circle (2cm);
	\filldraw[shaded] (0,0) circle (2cm);	
	\node at (0,0)  [empty box] (T) {};

	\draw[unshaded] (0,0) circle (1.2cm);

	\draw[ultra thick] (0,0) circle (2cm);
	\node at (0,0)  [empty box] (T) {};
	\node at (T.180) [left] {$*$};		
\end{tikzpicture}
\hs,\hs
\begin{tikzpicture}[annular]
	\clip (0,0) circle (2cm);
	\draw[ultra thick] (0,0) circle (2cm);	
	\node at (0,0)  [empty box] (T) {};
	\node at (180:1.80cm) [right] {$*$};

	\filldraw[shaded] (0,0) circle (1.2cm);

	\draw[ultra thick] (0,0) circle (2cm);
	\node at (0,0)  [empty box] (T) {};
\end{tikzpicture}
$$
\caption{Type II tangles $s_+,s_-$}
\end{figure}
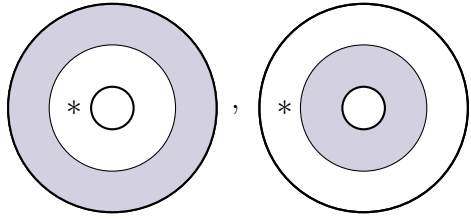

\begin{rem}
For the case $a_i b_j\colon[0]\to[0]$ (where we do not specify $\pm$), a suitable version of relation (6) reads
$$
a_i b_j=
\begin{cases}
s_- &\hsp{if} (i,j)=(1,2)\\
(\id_{[0+]},1,0) &\hsp{if} i=j=1 \\
(\id_{[0-]},0,1) &\hsp{if} i=j=2 \\
s_+ &\hsp{if} (i,j)=(2,1).
\end{cases}
$$ 
Note that we replace $t^{\pm 1}$ with $s_\pm$, which supports Graham and Lehrer's reasoning that the rotation converges to the non-contractible loop as $n\to 0$ in \cite{MR1659204}.
\end{rem}

\begin{lem}\label{typeinvolution} Let $m,n\in\N\cup\{0\pm\}$. Types are related to the involution as follows: 
\be
\item[(1)] $T\in\T_{I}(m,n)$ if and only if $T^*\in\T_{III}(m,n)$, and
\item[(2)]  If $T\in \T_{II}(n,n)$, then $T^*\in\T_{II}(n,n)$.
\ee
\end{lem}
\begin{proof}
Obvious.
\end{proof}

\begin{prop}\label{capindex} Let $m,n\in\N\cup\{0\pm\}$.
\be
\itt{Type I} Any $T\in\T_I(m,n)$ is uniquetly determined by $\capind(T)$. Moreover, $\rel_*(T)\in\{0,+(0),-(0)\}$.
\itt{Type II} Suppose $m=n\in\N$ or $m,n\in\{0+,0-\}$. Any $T\in\T_{II}(m,n)$ is uniquely determined by $\rel_*(T)$.
\itt{Type III} Any $T\in\T_{III}(m,n)$ is uniquely determined by $\cupind(T)$. Moreover, $\rel_*(T)\in\{0,+(0),-(0)\}$.
\ee
\end{prop}
\begin{proof}
\itt{Type I} Suppose $T_1,T_2\in\T_1(m,n)$ with $\capind(T_1)=\capind(T_2)$. If $\Lambda_i\in\caps(T_i)$ for $i=1,2$, note that $|B(\Lambda_1)|=|B(\Lambda_2)|$, so the $\Lambda_i$'s must end at the same points. Hence all caps of $T_i$ start and end at the same points for $i=1,2$. Now note that all other points on $D_1(T_i)$ for $i=1,2$ (if there are any) are connected to through strings, and recall $*_0(T_i)$ is uniquely determined by $*_1(T_i)$ for $i=1,2$. Hence $T_1=T_2$. The statement about $\rel_*(T)$ follows immediately from conditions (I-1) and (I-2).
\itt{Type II} Note that exactly one of the following occurs:
\be
\item[(1)] $m=n$ and $T=\id_{[n]}$, in which case $\rel_*(T)\in\{0,+(0),-(0)\}$,
\item[(2)] $m=n$ and $T=t^k$ where $0<k<n$, in which case $\rel_*(T)=k$,
\item[(3)] $m=n=0\pm$ and $T=(s_\mp s_\pm)^k$ for some $k\in\N$, in which case $\rel_*(T)=\pm(2k)$, or
\item[(4)] $m=0\pm$ and $n=0\mp$ and $T=(s_\pm s_\mp)^k s_\pm$ for some $k\in\Z_{\geq 0}$, in which case $\rel_*(T)=\pm(2k+1)$.
\ee
\itt{Type III} This follows immediately from the Type I case and Lemma \ref{typeinvolution}.
\end{proof}

\begin{lem}\label{compose}
Tangle type is preserved under tangle composition for tangles.
\end{lem}
\begin{proof}
\itt{Type I} Suppose $S,T\in\T_I$ such that $R=S\circ T$ makes sense. Certainly $R$ has no cups or loops. It remains to verify that $*_0(R)$ is in the right place. A problem could only arise in the case where both $S$ and $T$ have through strings, but we see that if $S$ and $T$ both satisfy condition (I-2), then so does $R$.
\itt{Type II} Obvious.
\itt{Type III} Suppose $S,T\in\T_{III}$ such that $R=S\circ T$ makes sense. Then by Lemma \ref{typeinvolution}, we have $T^*,S^*\in\T_{I}$ and $R^*=T^*\circ S^*$ makes sense, so by the Type I case, $R^*\in\T_I$, and once more by \ref{typeinvolution}, $R\in\T_{III}$.
\end{proof}

\begin{cor}\label{typecor} By  \ref{typerem} and Proposition \ref{compose},
\item[(1)] any composite of $a_i$'s is in $\T_I$, and
\item[(2)] any composite of $b_i$'s is in $\T_{III}$.
\end{cor}

\subsection{Unique Tangle Decompositions}
For this section, we use the convention that if $n=0\pm$ and $z\in\Z$, then $n+z=z$.

\begin{defn}
A tangle $T\in\T_{I}$ is called irreducible if there is at most one cap bounding $*_1(T)$, and if there is a cap $\Lambda$ bounding $*_1(T)$, then all other caps of $T$ are bounded by $\Lambda$.
\end{defn}

\begin{rem}\label{adecomp}
If $T\in\T_{I}(m,n)$ for $m\geq 1$ is irreducible, then $T$ has a unique representation as follows:
\itt{Case 1} if there is no cap bounding $*_1(T)$, then $T=a_{i_k}\cdots a_{i_1}$ with $i_j>i_{j+1}$ for all $j\in\{1,\dots,k-1\}$ and $i_j<2(m-j)+2$ for all $j\in\{1,\dots,k\}$.
\itt{Case 2} If there is a cap bounding $*_1(T)$, then $T=a_qa_{i_k}\cdots a_{i_1}a_{j_l}\cdots a_{j_1}$ where $k,l\geq 0$ and
\be
\item[(i)] $q=2n+2$,
\item[(ii)] $i_r>i_{r+1}$ for all $r\in\{1,\dots,k-1\}$, $i_1<j_l$, and $j_s>j_{s+1}$ for all $s\in\{1,\dots,l-1\}$, and
\item[(iii)] $i_r\leq 2(k-r)+1$ for all $r\in\{1,\dots,k\}$ and $j_s\geq 2(m-s)+1$ for all $s\in \{1,\dots,l\}$.
\ee
Uniqueness follows by looking at the cap indices which are given as follows:
\itt{Case 1} If there is no cap bounding $*_1(T)$, then $\capind(T)=\{i_{k},\cdots, i_1\}$.
\itt{Case 2} If there is a cap $\Lambda$ bounding $*_1(T)$, then $\ind(\Lambda)=2(m-l)$ and $\capind(T)=\{i_{k},\cdots, i_1,2(m-l),j_l,\cdots,j_1\}$.
\end{rem}

\begin{rems}
Suppose $T\in \T_{I}(m,n-1)$ with $m>n-1\geq 1$ is irreducible such that $*_1(T)$ is bounded. Let
$T=a_qa_{i_k}\cdots a_{i_1}a_{j_l}\cdots a_{j_1}$ be the representation afforded by the above remark. If $S\in\T_I(n-1,p)$ and $R=S\circ T$, then
\item[(1)] there is a cap $\Lambda$ of $R$ bounding $*_1(R)$, of index $2(m-l)$. All other caps of $R$ bounding $*_1(R)$ have smaller index than $\Lambda$.
\item[(2)] $|B(\Lambda)|=k+l+1$.
\item[(3)] $\capind(R)=\{i_{k},\cdots,i_1,c_1,\cdots,c_s,2(m-l),j_l,\dots,j_1\}$ for some $c_1,\dots,c_s\in\N$ and $s=m-p-k-l-1$. 
\end{rems}

\begin{figure}[!ht]
\begin{tikzpicture}
	\clip (1,11) rectangle (-5,13.5);
	\draw[shaded] (0,0) circle (15cm);	
	\node at (0,0)  [empty box] (T) {};
	\node at (135:1.80cm) [above] {$*$};
	\filldraw[unshaded] (90:12cm) .. controls ++(90:2cm) and ++(110:2cm) .. (110:12cm);
	\filldraw[shaded] (108:12cm) .. controls ++(108:8mm) and ++(104:8mm) .. (104:12cm);
	\filldraw[shaded] (102:12cm) .. controls ++(102:1cm) and ++(92:1cm) .. (92:12cm);	\filldraw[unshaded] (100:12cm) .. controls ++(100:5mm) and ++(98:5mm) .. (98:12cm);
	\filldraw[unshaded] (96:12cm) .. controls ++(96:5mm) and ++(94:5mm) .. (94:12cm);
	
	\node at (103:12cm) [above] {$*$};
	\draw[ultra thick] (0,0) circle (12cm);	
	\draw[unshaded] (0,0) circle (12cm);	
\end{tikzpicture}\\
\caption{$R=S\circ T$, zoomed in near $*_1(R)$ where $T=a_{2n+2}a_1a_2a_4a_{2m-1}\in\T_{I}(m,n)$ is irreducible}\label{irredfigure}
\end{figure}
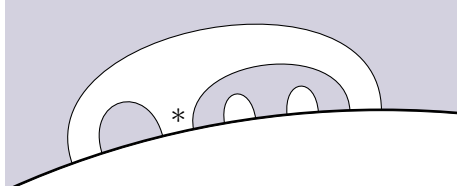

\begin{lem}\label{irred}
Suppose $T_1\in\T_I(m,m-u-1)$ and $T_2\in\T_I(m,m-v-1)$ with $m-u,m-v\geq 2$ are irreducible and each has one cap bounding $*_1$. Suppose $S_1\in\T_I(m-u-1,w)$ and $S_2\in \T_I(m-v-1,w)$ such that $S_1\circ T_1=S_2\circ T_2$. Then $T_1=T_2$.
\end{lem}
\begin{proof}
Set $R=S_1\circ T_1=S_0\circ T_0$. We have that $*_1(R)$ is bounded by a cap $\Lambda$ with index $2(m-u)=2(m-v)$, so $u=v$. Now we have unique irreducible decompositions
\begin{align*}
T_1&=a_pa_{i_k}\cdots a_{i_1}a_{j_l}\cdots a_{j_1}\hsp{and}\\
T_2&=a_qa_{g_r}\cdots a_{g_1}a_{h_s}\cdots a_{h_1},
\end{align*}
and as the cap indices of $R$ are unique, we have
\begin{align*}
\capind(R)&=\{i_{k},\cdots,i_1,c_1,\cdots,c_s,2(m-u),j_l\cdots, j_1\}\\
&=\{g_r,\cdots,g_1,c_1,\dots,c_s,2(m-v),h_s,\cdots,h_1\}.\end{align*}
Hence we must have equality of the two sequences:
$$\{i_{k},\cdots,i_1,2(m-u),j_l\cdots, j_1\}=\{g_r,\cdots,g_1,2(m-v),h_s,\cdots,h_1\},$$
and $T_1=T_2$ by Proposition \ref{capindex}.
\end{proof}

\begin{prop}\label{a}
Each $T\in\T_I(m,n)$ where $m\in\N$ and $n\in\N\cup\{0\pm\}$ has a unique decomposition $T=W_r\cdots W_1$ such that $W_i$ is irreducible for all $i=1,\dots,r$.
\end{prop}
\begin{proof}
\itt{Existence} 
The existence of such a decomposition will follow from Algorithm \ref{readtangle} below.
\itt{Uniqueness}
We induct on $r$. Suppose $r=1$. Then uniqueness follows from Remark \ref{adecomp}. Suppose now that $r>1$ and the result holds for all concatenations of fewer irreducible words. Suppose we have another decomposition
$$
T=W_r\cdots W_1=U_s\cdots U_1.
$$
Then by the induction hypothesis, we must have $s\geq r$. As $W_1$ and $U_1$ are irreducible, we apply Lemma \ref{irred} with
\be
\item[(1)] $T_1=W_1$ and $S_1=W_r\cdots W_2$, and
\item[(2)] $T_2=U_1$ and $S_2=U_s\cdots U_2$
\ee
to see that $W_1=U_1$. We may now apply appropriate $b_i$'s to $T$ (on the right) to get rid of $W_1=U_1$ to get
$$
W'=W_r\cdots W_2=U_s\cdots U_2.
$$ 
where $W'$ is equal to a concatenation of fewer irreducible words. By the induction hypothesis, we can conclude $r=s$ and $U_i=W_i$ for all $i=2,\dots, r$. We are finished.
\end{proof}

\begin{alg}\label{readtangle}
The following algorithm expresses a Type I tangle $T\in\T_I(m,n)$ where $m\in\N$ and $n\in\N\cup\{0\pm\}$  as a composite of $a_i$'s in the form required by Proposition \ref{a}. Set $T_0=T$, $m_0=m$, and $r=1$.
\be
\itt{Step 1} Let $S_1=\set{\Lambda\in\caps(T_0)}{*_1(T_0)\subset \partial \Lambda\hsp{and}\ind(\Lambda)\in2\N}$. Let $S_0$ be the set of all caps that are not in $B(\Lambda)$ for some $\Lambda\in S_1$. If $S_1=\emptyset$, proceed to Step 4.
\itt{Step 2} Suppose $|S_1|\geq 1$. Select the cap $\Lambda\in S_1$ with the largest index. There are two cases:
\be
\itt{Case 1} $B(\Lambda)=\{\Lambda\}$. Set $W_r=\alpha_{\ind(\Lambda)}$. Proceed to Step $3$.
\itt{Case 2} $B(\Lambda)\setminus\{\Lambda\}\neq \emptyset$. List the cap indices for all caps $\Lambda'\in B(\Lambda)\setminus\{\Lambda\}$ in decreasing order from right to left,  $i_k,\cdots,i_1$ where $i_j> i_{j+1}$ for all $j\in\{1,\dots,k-1\}$. where $k=|B(\Lambda)\setminus\{\Lambda\}|$. Set $q=\ind(\Lambda)-2k$ and $W_r=a_qa_{i_k}\cdots a_{i_1}$.
\ee
\itt{Step 3} Note that $W_r$ is irreducible. Move $*_1(T_0)$ counterclockwise to the closest simple interval outside of $\Lambda$ whose interior touches an unshaded region (which is necessarily $2$ regions counterclockwise), and remove all caps in $B(\Lambda)$ from $T_0$ to get a new tangle, called $T_1$.  Note that $T_0=T_1W_r$. Set $m_1$ equal to half the number of internal boundary points of $T_1$, and set $r_1=r$. Now set $T_0=T_1$, $m_0=m_1$, and $r=r_1+1$. Go back to Step 1.
\itt{Step 4} List the cap indices for all caps $\Lambda\in S_0$ in decreasing order from right to left,  $i_k,\cdots,i_1$ where $i_j> i_{j+1}$ for all $j\in\{1,\dots,k-1\}$.  There are two cases:
\be
\item[(i)] There are fewer than $m_0$ caps. Set $W_r=a_{i_k}\cdots a_{i_1}$. Note that $W_r$ is irreducible and $T_0=W_r$. We are finished.
\item[(ii)] There are $m_0$ caps. Proceed to Step $5$.
\ee
 \itt{Step 5}
There are two cases:
\be
\item[(i)]
If the region touching $D_0(T_0)$ is unshaded, set $W_r=a_1a_{i_{k-1}}\cdots a_{i_1}$. Note that $W_r$ is irreducible and $T_0=W_r$. We are finished.
\item[(ii)] 
If the region touching $D_0(T_0)$ is shaded, set $W_r=a_2a_{i_{k-1}}\cdots a_{i_1}$. Note that $W_r$ is irreducible and $T_0=W_r$. We are finished.
\ee
\ee
Note that $T=W_r\cdots W_1$ satisfies the conditions of Proposition \ref{a}.
\end{alg}

The following Theorem is merely a strengthening of Corollary 1.16 in \cite{MR1309131}.
\begin{thm}[Atl Tangle Decomposition]\label{decomposition}
Each Atl $(m,n)$-tangle $T$ can be written uniquely as a composite $T=T_{III}\circ T_{II}\circ T_I$ where $T_i\in\T_i$ for all $i\in\{I,II,III\}$.
\end{thm}
\begin{proof} We begin by proving the uniqueness of such a decomposition as it will tell us how to find such a decomposition.
\itt{Uniqueness} Suppose we have a decomposition $T=T_{III}\circ T_{II}\circ T_I$ where $T_I\in\T_I(m,l)$, $T_{II}\in\T_{II}(l,k)$, and $T_{III}\in \T_{III}(k,n)$ for some $l,k\in\N\cup\{0\pm\}$. Note that $l,k$ are uniquely determined by $|\ts(T)|$ and the shading of $T$. Note further that $\capind(T_I)=\capind(T)$, $\rel_*(T_{II})=\rel_*(T)$, and $\cupind(T_{III})=\cupind(T)$. Hence $T_i$ is uniquely determined for $i\in\{I,II,III\}$ by Proposition \ref{capindex}.
\itt{Existence}
Let $l=k$ be the number of through strings of $T$. If $l=k=0$, set $l=0+$, respectively $l=0-$ if the region meeting $D_1(T)$ is unshaded, respectively shaded, and set $k=0+$, respectively $k=0-$ if the region meeting $D_0(T)$ is unshaded, respectively shaded. Let $T_I\in\T_I(m,l)$ be the unique tangle with $\capind(T_I)=\capind(T)$. Let $T_{II}\in\T_{II}(l,k)$ be the unique tangle with $\rel_*(T_{II})=\rel_*(T)$. Let $T_{III}\in\T_{III}(k,n)$ be the unique tangle with  $\cupind(T_{III})=\cupind(T)$. It is now obvious that $T=T_{III}\circ T_{II}\circ T_I$.  
\end{proof}

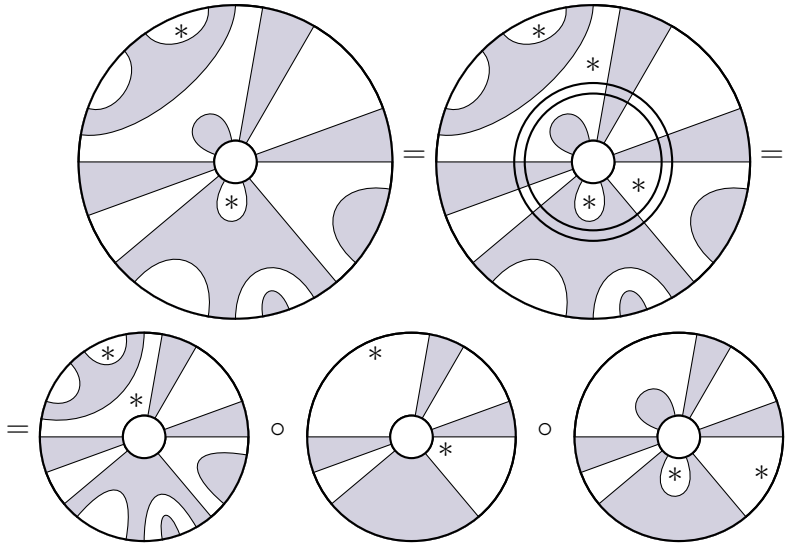
\begin{figure}[!ht]
\begin{align*}
\begin{tikzpicture}[PAdefn]
	\clip (0,0) circle (3cm);
	\draw[ultra thick] (0,0) circle (3cm);	
	\node at (0,0)  [empty box] (T) {};
	\filldraw[shaded] (0:4cm)--(0,0)--(20:4cm);	
	\filldraw[shaded] (60:4cm)--(0,0)--(80:4cm);
	\filldraw[shaded] (180:4cm)--(0,0)--(200:4cm);
	\filldraw[shaded] (220:5cm)--(0,0)--(310:5cm);		
	\filldraw[shaded] (T.100) .. controls ++(100:11mm) and ++           (160:11mm) .. (T.160);
	\filldraw[unshaded] (T.240) .. controls ++(240:11mm) and ++           (290:11mm) .. (T.290);
	\filldraw[shaded] (90:4cm)--(90:3cm) .. controls ++(270:11mm) and ++(350:11mm) .. (170:3cm) -- (170:4cm);
	\filldraw[unshaded] (100:4cm)--(100:3cm)  .. controls ++(280:6mm) and ++(305:6mm) ..         (125:3cm) -- (125:4cm);
	\filldraw[unshaded] (135:4cm)--(135:3cm)  .. controls ++(315:6mm) and ++(340:6mm) ..         (160:3cm) -- (160:4cm);
	\filldraw[unshaded] (270:4cm)--(270:3cm)  .. controls ++(90:11mm) and ++(120:11mm) ..         (300:3cm) -- (300:4cm);
	\filldraw[shaded] (280:4cm)--(280:3cm)  .. controls ++(100:6mm) and ++(110:6mm) .. (290:3cm) -- (290:4cm);
	\filldraw[unshaded] (230:4cm)--(230:3cm)  .. controls ++(50:11mm) and ++(80:11mm) ..         (260:3cm) -- (260:4cm);
	\filldraw[shaded] (320:4cm)--(320:3cm)  .. controls ++(140:11mm) and ++(170:11mm) ..         (350:3cm) -- (350:4cm);			
	\draw[ultra thick] (0,0) circle (3cm);	
	\node at (0,0)  [empty box] (T) {};
	\node at (T.260) [below] {$*$};
	\node at (110:3cm)[below] {$*$};
\end{tikzpicture}=
\begin{tikzpicture}[PAdefn]
	\clip (0,0) circle (3cm);
	\draw[ultra thick] (0,0) circle (3cm);	
	\node at (0,0)  [empty box] (T) {};
	\filldraw[shaded] (0:4cm)--(0,0)--(20:4cm);	
	\filldraw[shaded] (60:4cm)--(0,0)--(80:4cm);
	\filldraw[shaded] (180:4cm)--(0,0)--(200:4cm);
	\filldraw[shaded] (220:5cm)--(0,0)--(310:5cm);		
	\filldraw[shaded] (T.100) .. controls ++(100:11mm) and ++           (160:11mm) .. (T.160);
	\filldraw[unshaded] (T.240) .. controls ++(240:11mm) and ++           (290:11mm) .. (T.290);
	\filldraw[shaded] (90:4cm)--(90:3cm) .. controls ++(270:11mm) and ++(350:11mm) .. (170:3cm) -- (170:4cm);
	\filldraw[unshaded] (100:4cm)--(100:3cm)  .. controls ++(280:6mm) and ++(305:6mm) ..         (125:3cm) -- (125:4cm);
	\filldraw[unshaded] (135:4cm)--(135:3cm)  .. controls ++(315:6mm) and ++(340:6mm) ..         (160:3cm) -- (160:4cm);
	\filldraw[unshaded] (270:4cm)--(270:3cm)  .. controls ++(90:11mm) and ++(120:11mm) ..         (300:3cm) -- (300:4cm);
	\filldraw[shaded] (280:4cm)--(280:3cm)  .. controls ++(100:6mm) and ++(110:6mm) .. (290:3cm) -- (290:4cm);
	\filldraw[unshaded] (230:4cm)--(230:3cm)  .. controls ++(50:11mm) and ++(80:11mm) ..         (260:3cm) -- (260:4cm);
	\filldraw[shaded] (320:4cm)--(320:3cm)  .. controls ++(140:11mm) and ++(170:11mm) ..         (350:3cm) -- (350:4cm);			
	\draw[ultra thick] (0,0) circle (3cm);
	\draw[thick] (0,0) circle (1.5cm);
	\draw[thick] (0,0) circle (1.3cm);	
	\node at (0,0)  [empty box] (T) {};
	\node at (T.260) [below] {$*$};
	\node at (90:1.5cm) [above] {$*$};
	\node at (-20:1.3cm) [left] {$*$};	
	\node at (110:3cm)[below] {$*$};
\end{tikzpicture}=\\
=
\begin{tikzpicture}[annular]
	\clip (0,0) circle (2cm);
	\draw[ultra thick] (0,0) circle (2cm);	
	\node at (0,0)  [empty box] (T) {};
	\filldraw[shaded] (0:4cm)--(0,0)--(20:4cm);	
	\filldraw[shaded] (60:4cm)--(0,0)--(80:4cm);
	\filldraw[shaded] (180:4cm)--(0,0)--(200:4cm);
	\filldraw[shaded] (220:5cm)--(0,0)--(310:5cm);		
	\filldraw[shaded] (90:4cm)--(90:2cm) .. controls ++(270:11mm) and ++(350:11mm) .. (170:2cm) -- (170:4cm);
	\filldraw[unshaded] (100:4cm)--(100:2cm)  .. controls ++(280:6mm) and ++(305:6mm) ..         (125:2cm) -- (125:4cm);
	\filldraw[unshaded] (135:4cm)--(135:2cm)  .. controls ++(315:6mm) and ++(340:6mm) ..         (160:2cm) -- (160:4cm);
	\filldraw[unshaded] (270:4cm)--(270:2cm)  .. controls ++(90:11mm) and ++(120:11mm) ..         (300:2cm) -- (300:4cm);
	\filldraw[shaded] (280:4cm)--(280:2cm)  .. controls ++(100:6mm) and ++(110:6mm) .. (290:2cm) -- (290:4cm);
	\filldraw[unshaded] (230:4cm)--(230:2cm)  .. controls ++(50:11mm) and ++(80:11mm) ..         (260:2cm) -- (260:4cm);
	\filldraw[shaded] (320:4cm)--(320:2cm)  .. controls ++(140:11mm) and ++(170:11mm) ..         (350:2cm) -- (350:4cm);			
	\draw[ultra thick] (0,0) circle (2cm);	
	\node at (0,0)  [empty box] (T) {};
	\node at (T.110) [above] {$*$};
	\node at (110:2cm)[below] {$*$};
\end{tikzpicture}
\hs \circ \hs
\begin{tikzpicture}[annular]
	\clip (0,0) circle (2cm);
	\draw[ultra thick] (0,0) circle (2cm);	
	\node at (0,0)  [empty box] (T) {};
	\filldraw[shaded] (0:4cm)--(0,0)--(20:4cm);	
	\filldraw[shaded] (60:4cm)--(0,0)--(80:4cm);
	\filldraw[shaded] (180:4cm)--(0,0)--(200:4cm);
	\filldraw[shaded] (220:5cm)--(0,0)--(310:5cm);		
	\draw[ultra thick] (0,0) circle (2cm);	
	\node at (0,0)  [empty box] (T) {};
	\node at (T.-35) [right] {$*$};
	\node at (110:2cm)[below] {$*$};
\end{tikzpicture}
\hs \circ\hs
\begin{tikzpicture}[annular]
	\clip (0,0) circle (2cm);
	\draw[ultra thick] (0,0) circle (2cm);	
	\node at (0,0)  [empty box] (T) {};
	\filldraw[shaded] (0:4cm)--(0,0)--(20:4cm);	
	\filldraw[shaded] (60:4cm)--(0,0)--(80:4cm);
	\filldraw[shaded] (180:4cm)--(0,0)--(200:4cm);
	\filldraw[shaded] (220:5cm)--(0,0)--(310:5cm);		
	\filldraw[shaded] (T.100) .. controls ++(100:11mm) and ++           (160:11mm) .. (T.160);
	\filldraw[unshaded] (T.240) .. controls ++(240:11mm) and ++           (290:11mm) .. (T.290);
	\draw[ultra thick] (0,0) circle (2cm);	
	\node at (0,0)  [empty box] (T) {};
	\node at (T.260) [below] {$*$};
	\node at (-20:2cm)[left] {$*$};
\end{tikzpicture}
\end{align*}
\caption{Decomposition of an ATL tangle into $T_{III}\circ T_{II}\circ T_I$}
\end{figure}

\section{The Category $\aD$}\label{aD}

\subsection{Generators and Relations}

\begin{defn}
Let $\aD$, the annular category, be the following small category:
\itt{Objects} $[n]$ for $n\in\N\cup{0\pm}$, and
\itt{Morphisms} generated by 
\begin{align*}
\alpha_1\colon [1]&\rightarrow [0+],\hs \alpha_2\colon [1]\rightarrow [0-],\hsp{and}\\
\alpha_i\colon [n]&\longrightarrow [n-1] \hsp{for} i=1,\dots,2n\hsp{and} n\geq 2;\\
\beta_1\colon [0+]&\rightarrow [1],\hs\beta_2\colon [0-]\rightarrow [1],\hsp{and}\\
\beta_i\colon [n]&\longrightarrow [n+1] \hsp{for} i=1,\dots,2n+2\hsp{and} n\geq 1;\\
\tau\colon [n] &\longrightarrow [n] \hsp{for all} n\in\N;\hsp{and}\\
\delta_{\pm}\colon [n]&\longrightarrow [n]\hsp{for all} n\in\N\cup\{0\pm\}
\end{align*}
subject to the following relations:
\be
\item[(1)] $\alpha_i\alpha_j = \alpha_{j-2}\alpha_i$ for $i<j-1$ and $(i,j)\neq (1,2n)$,
\item[(2)] $\beta_i\beta_j=\beta_{j+2}\beta_i$ for $i\leq j$ and $(i,j)\neq (1,2n+2)$,
\item[(3)] $\tau^n=\id_{[n]}$,
\item[(4)] $\alpha_i\tau = \tau\alpha_{i-2}$ for $i\geq 3$, 
\item[(5)] $\beta_i \tau = \tau\beta_{i-2}$ for $i\geq 3$,
\item[(6)] $\delta_+=\alpha_1\beta_1\in \aD(0+,0+)$ and $\delta_-=\alpha_2\beta_2\in \aD(0-,0-)$. If $\alpha_i\beta_j\colon [n]\to [n]$ with $n\geq 1$, then 
$$\D
\alpha_i\beta_j=
\begin{cases}
\tau^{-1} &\hsp{if} (i,j)=(1,2n+2)\\
\beta_{j-2}\alpha_i &\hsp{if} i<j-1, (i,j)\neq (1,2n+2)\\
\id_{[n]} &\hsp{if} i=j-1\\
\delta_+ &\hsp{if} i=j\hsp{and $i$ is odd} \\
\delta_- &\hsp{if} i=j\hsp{and $i$ is even} \\
\id_{[n]} &\hsp{if} i=j+1\\
\beta_{j}\alpha_{i-2} &\hsp{if} i>j+1, (i,j)\neq (2n+2,1)\\
\tau &\hsp{if} (i,j)=(2n+2,1)
\end{cases}$$
\item[(7)] $\delta_\pm$ commutes with all other generators (including $\delta_\mp$).
\ee
\end{defn}

\subsection{Involution and Word Type}

\begin{defn}
A morphism $h\in\Mor(\aD)$ will be called primitive if $h$ is equal to $\alpha_i$, $\beta_i$, $t$, $\delta_\pm$, or $\id_{[n]}$ for $n\in\N\cup\{0\pm\}$. A word on $\aD$ is a sequence $h_r\cdots h_1$ with $r\geq 1$ of primitive morphisms in $\aD$. We say the length of such a word is $r\in\N$. By convention, we will say a word has length zero if and only if $r=1$ and $h_1=\id_{[n]}$ for some $n\in\N\cup\{0\pm\}$.
\end{defn}

\begin{defn}
We define a map $*$ on $\Ob(\aD)$ and on primitive morphisms in $\Mor(\aD)$:
\be
\item[(Ob)] For $n\in\N\cup\{0\pm\}$, define $[n]^*=[n]$.
\item[(I)] For all $n\in\N\cup \{0\pm\}$, define $\id_{[n]}^*=\id_{[n]}$.
\item[(A)] For $\alpha_1\in\aD(1,0+)$, define $\alpha_1^*=\beta_1\in\aD(0+,1)$. For $\alpha_2\in\aD(1,0-)$, define $\alpha_2^*=\beta_2\in\aD(0-,1)$. For $n\geq 2$ and $\alpha_i\in\aD(n,n-1)$, so $i\in\{1,\dots,2n\}$, define $\alpha_i^*=\beta_i\in\aD(n-1,n)$.
\item[(B)] For $\beta_1\in\aD(0+,1)$, define $\beta_1^*=\alpha_1\in\aD(1,0+)$. For $\beta_2\in\aD(0-,1)$, define $\beta_2^*=\alpha_2\in\aD(1,0-)$. For $n\geq 1$ and $\beta_i\in\aD(n,n+1)$, so $i\in\{1,\dots,2n+2\}$, define $\beta_i^*=\alpha_i\in\aD(n+1,n)$.
\item[(T)] For $n\in\N$ and $\tau\in\aD(n,n)$, define $\tau^*=\tau^{-1}$.
\item[(D)] For $n\in\N\cup\{0\pm\}$ and $\delta_{\pm}\in \aD(n,n)$, define $\delta_\pm^*=\delta_\pm$.
\ee
\end{defn}

\begin{prop}\label{involution}
The following extension of $*$ to $\Mor(\aD)$ is well defined:
\be
\item[$\bullet$] If $h_r\cdots h_1$ is a word on $\aD$, then we define $(h_r\cdots h_1)^*=h_1^*\cdots h_r^*$.
\ee
\end{prop}
\begin{proof}
We must check that $*$ preserves the relations of $\aD$. Note that relations (3), (6), and (7) are preserved by $*$, and the following pairs are switched: (1) $\&$ (2) and (4) $\&$ (5).
\end{proof}

\begin{defn}
By Proposition \ref{involution}, $*$ extends uniquely to an involution on $\aD$.
\end{defn}
\begin{cor}
We have an isomorphism of categories $\aD\cong \aD\op$.
\end{cor}

\begin{prop}\label{additional} The following additional relations hold in $\aD$:
\be
\item[(1)] $\alpha_1\tau = \alpha_{2n-1}$ and $\alpha_2 \tau = \alpha_{2n}$,
\item[(2)] $\tau\beta_{2n+1}=\beta_1$, $\tau\beta_{2n+2}=\beta_{2}$, and
\item[(3)] $\beta_1 \tau = \tau^2\beta_{2n-1}$ and $\beta_2 \tau = \tau^2\beta_{2n}$.
\ee
\end{prop}
\begin{proof}
\item[(1)]
By relations (4) and (5), we have 
$$\alpha_{2n-1}=\alpha_{2n-1}\tau^n = \tau\alpha_{2n-3} \tau^{n-1}=\cdots = \tau^{n-1} \alpha_{3} \tau = \tau^n\alpha_1 = \alpha_1.
$$
The proof of the other relation is similar.
\item[(2)]
These relations are merely $*$ applied to (1).
\item[(3)]
By relations (4) and (6), we have
$$
\tau^2\beta_{2n-1} =\tau^2\beta_{2n-1} \tau^{n}=\tau^2 \tau \beta_{2n-3}\tau^{n-1}=\cdots=\tau^{2} \tau^{n-1}\beta_1\tau = \tau^{n+1}\beta_1\tau = \beta_1\tau.
$$
The proof of the other relation is similar.
\end{proof}

\begin{nota}
\item[(1)]
If $h\in \Mor(\aD)$, we write $h\in \A_1$ if $h=\alpha_i\in \aD(1,0\pm)$ where $i\in\{1,2\}$. We write $h\in \A_n$ where $n\geq 2$ if $h=\alpha_i\in\aD(n,n-1)$ for some $i\in\{1,\dots,2n\}$. We write $h\in \A$ if $h\in \A_n$ for some $n\geq 1$. Similarly we define $\B_n$ for $n\in\N\cup\{0\pm\}$ and $\B$.
\item[(2)] For convenience, we will use the notation $\sigma_+=\alpha_2\beta_1\in \aD(0+,0-)$ and $\sigma_-=\alpha_1\beta_2\in \aD(0-,0+)$.
\end{nota}

\begin{defn}
A word $w=h_r\cdots h_1$ on $\aD$ is called
\be
\itt{Type I} if $w$ has length zero or if $h_i\in \A$ for all $i\in\{1,\dots,r\}$.
\itt{Type II} if either
\be
\item[(1)] $w$ has length zero,
\item[(2)] $r>0$ and $h_i=\tau$ for all $i\in\{1,\dots,r\}$, or 
\item[(3)] $r=2s$ for some $s>0$ and $h_ih_{i+1}=\sigma_\pm$ for all odd $i$ so that
$$w=\begin{cases}
(\sigma_\pm\sigma_\mp)^k\sigma_\pm &\hsp{if $s=2k+1$ is odd, or}\\
(\sigma_\pm\sigma_\mp)^k &\hsp{if $s=2k$ is even.}
\end{cases}$$
\ee
\itt{Type III} if $w$ has length zero or if $h_i\in \B$ for all $i\in\{1,\dots,r\}$.
\ee
Denote the set of all words of Type i by $\W_i$, and denote the set of all words of Type $i$ with domain $[m]$ and codomain $[n]$ by $\W_i(m,n)$  for $i\in\{I,II,III\}$.
\end{defn}

\begin{lem} Let $m,n\in\N\cup\{0\pm\}$. Types are related to the involution as follows:
\be
\item[(1)] $w\in\W_I(m,n)$ if and only if $w^*\in \W_{III}(n,m)$, and
\item[(2)] If $w\in\W_{II}(n,n)$, then $w^*\in\W_{II}(n,n)$.
\ee
\end{lem}
\begin{proof}
Obvious.
\end{proof}

\subsection{Standard Forms}

\begin{nota}
if we replace $j$ with $j+2$ in the statement of relation (1), we get the equivalent relation
\be
\item[(1')] $\alpha_j\alpha_i = \alpha_i\alpha_{j+2} \hsp{for all} j\geq i \hsp{with} (j,i)\neq (2n,1)$
\ee
as maps $[n+1]\to [n-1]$.
\end{nota}

\begin{defn}\label{irredalpha}
A word $w\in\W_I(m,n)$ with $m\geq 1$ is called irreducible if either
\item[(1)] $w=\alpha_{i_k}\cdots \alpha_{i_1}$ where $i_r>i_{r+1}$ for all $r\in\{1,\dots,k-1\}$ and $i_r<2(m-r)+2$ for all $r\in\{1,\dots,k\}$, in which case we also say $w$ is ordered, or
\item[(2)]
$w=\alpha_q\alpha_{i_k}\cdots\alpha_{i_1}\alpha_{j_l}\cdots\alpha_{j_1}\in\W_{I}(m,n)$ where $m\geq 1$ and $l,k\geq 0$ such that
\be
\item[(i)] $q=2n+2$,
\item[(ii)] $i_r>i_{r+1}$ for all $r\in\{1,\dots,k-1\}$, $i_1<j_l$, and $j_s>j_{s+1}$ for all $s\in\{1,\dots,l-1\}$, and
\item[(iii)] $i_r\leq 2(k-r)+1$ for all $r\in\{1,\dots,k\}$ and $j_s\geq 2(m-s)+1$ for all $s\in \{1,\dots,l\}$.
\ee
\end{defn}

\begin{rem}
If $\alpha_q\alpha_{i_k}\cdots\alpha_{i_1}\alpha_{j_l}\cdots\alpha_{j_1}$ is irreducible as in (2) of \ref{irredalpha}, then so are $$\alpha_q\alpha_{i_k}\cdots\alpha_{i_1}\alpha_{j_l}\cdots\alpha_{j_r}\hsp{and}\alpha_q\alpha_{i_k}\cdots\alpha_{i_s}$$ for all $r\in\{1,\dots,l\}$ and  $s\in\{1,\dots,k\}$. In particular, if $l>0$, then $j_l=2(m-l)+1$, and if $k>0$, then $i_k=1$.
\end{rem}

\begin{alg}\label{irredalg}
Suppose $w =\alpha_{i_k}\cdots \alpha_{i_1}\in\W_1(m,n-1)$ is ordered where $n-1>0$. The following algorithm gives words $u_1, u_2$ where $u_1$ is irreducible and $ \alpha_{2n} w=u_2 u_1$. Set $u_1=\alpha_{2n}w$ and $u_3=\id_{[n-1]}$.
\itt{Step 1} If $u_1$ is irreducible, set $u_2=u_3$. We are finished. Otherwise, proceed to Step 2.
\itt{Step 2} There is a $j\in \{1,\dots,k\}$ such that $2(k-j)+1<i_j<2(m-j)+1$. Pick $j$ minimal with this property. Use relation (1) to push $\alpha_{i_k}\cdots\alpha_{i_{j+1}}$ past $\alpha_{i_j}$ to get
$$
w=\alpha_{2n}\alpha_{i_j-2(k-j)+2}\alpha_{i_{k-1}}\cdots\alpha_{i_{j+1}}\alpha_{i_{j-1}}\cdots\alpha_{i_1}.
$$
Note that 
$$1<i_j-2(k-j)< 2(m-j)+1-2(k-j)=2(m-k)+1=2n+1,$$ 
as $m-k=n$, so we may use relation (1') to get
$$
\alpha_{i_j-2(k-j)}\alpha_{2n+2}\alpha_{i_{k-1}}\cdots\alpha_{j+1}\alpha_{j-1}\cdots\alpha_{i_1}.
$$
Set $u_2=\alpha_{i_j-2(k-j)+2} u_3$. Now set $u_3=u_2$. Set
$$
u_1=\alpha_{2n+2}\alpha_{i_{k-1}}\cdots\alpha_{j+1}\alpha_{j-1}\cdots\alpha_{i_1}.
$$
Go back to Step 1.
\end{alg}
\begin{proof}
We need only prove the above algorithm terminates. Note one of the $\alpha_i$'s increases in index each iteration, which cannot happen indefinitely. 
\end{proof}

\begin{prop}\label{alpha} Suppose $m\in\N$ and $n\in\N\cup\{0\pm\}$. Each $w\in\W_I(m,n)$ has a decomposition $w=w_r\cdots w_1$ where each $w_i\in W_I$ is irreducible. Such a decomposition of $w$ is called a standard decomposition of $w$.
\end{prop}
\begin{proof}
We induct on the length of $w$. If the length of $w$ is $1$, then we are finished. Suppose $w$ has length greater than $1$ and the result holds for all words of shorter length. Use relation (1') to get $w'=w_r\cdots w_1$ where each $w_i$ is ordered and for each $s\in\{1,\dots,r-1\}$. If $w_s=\alpha_{i_a}\cdots\alpha_{i_1}$ and $w_{s+1}=\alpha_{j_b}\cdots\alpha_{j_1}$, then $i_a=1$, $j_1=2k$, so $\alpha_{j_1}\alpha_{i_a}=\alpha_{2k}\alpha_1\in\aD(k+1,k-1)$ for some $k\geq 2$. There are two cases.
\itt{Case 1} $r=1$. Then $w=w_1$ is ordered, hence irreducible, and we are finished.
\itt{Case 2} Suppose $r>1$. As $w_2=\alpha_{i_a}\cdots\alpha_{i_1}$ where $\alpha_{i_1}=\alpha_{2k}\in\aD(k,k-1)$, we apply Algorithm \ref{irredalg} to the word $\alpha_{2k}w_1$ to obtain $u_1,u_2$ with $u_1$ irreducible such that $u_2u_1=\alpha_{2k} w_1$. We now note that $w=w'u_1$ where
$$w'=w_r\cdots w_3 \alpha_{i_a}\cdots \alpha_{i_2}$$
is a word of strictly smaller length. Applying the induction hypothesis to $w'$ gives us the desired result.
\end{proof}

\begin{thm}[Standard Forms]\label{standard}
Suppose $w=h_r\cdots h_1$ is a word on $\aD$ in $\aD(m,n)$ for $m,n\in\N\cup\{0\pm\}$. Then there is a decomposition $w=\delta_+^{c_+}\delta_-^{c_-} w_{III}w_{II} w_{I}$ where $w_i\in \W_i$ for all $i\in\{I,II,III\}$, $c_\pm\geq 0$, and $w_I$ and $w_{III}^*$ are in the form afforded by Proposition \ref{alpha}.
\end{thm}
\begin{proof}
Note that it suffices to find $v_i\in \W_{i}$ for $i\in\{I,II,III\}$ and $c_\pm\geq 0$ such that $w=\delta_+^{c_+}c_-^{c_-}v_{III}v_{II}v_I$, as we may then set $w_{II}=v_{II}$ and apply Proposition \ref{alpha} to $v_I$ and $v_{III}^*$ to get $w_I$ and $w_{III}^*$ respectively. We induct on $r$. The case $r=1$ is trivial. Suppose $r>1$ and the result holds for all words of shorter length. Apply the induction hypothesis to $w'=h_{r-1}\cdots h_1$ to get
$$
w'=\delta_+^{c'_+}\delta_-^{c'_-}u_{III}u_{II}u_I.
$$
There are $6$ cases.
\item[(D)] Suppose $h_r=\delta_\pm$. Set $c_\pm=c'_\pm+1$, $c_\mp=c'_\mp$, and $v_i=u_i$ for all $i\in\{I,II,III\}$. We are finished.
\item[(B)] Suppose $h_r\in B$. Set $c_\pm=c'_\pm$ and $v_i=u_i$ for $i\in\{I,II,III\}$. We are finished.
\item[(T)] Suppose $h_r=\tau$. Set $c_\pm=c'_\pm$ and $w_I=u_I$. As we push $\tau$ right using relation (5) and Proposition \ref{additional}, only two extraordinary possibilities occur:
\be
\itt{Case 1} $\tau$ meets $\beta_{2n+1}$ or $\beta_{2n+2}$ in $\aD(n,n+1)$, so $\tau$ disappears when using Proposition \ref{additional}, or
\itt{Case 2} $\tau$ meets $\beta_1\in\aD(0+,1)$ or $\beta_2\in\aD(0-,1)$,  so $\tau$ disappears as $\id_{[1]}=\tau\in\aD(1,1)$.
\ee
Hence we get that $w=v_{III}'\tau^s$ where $v_{III}\in\W_{III}$ and $s\in\{0,1\}$. If $s=0$, set $v_{II}=u_{II}$, and if $s=1$, set $v_{II}=\tau u_{II}$. We are finished.
\item[(A)] Suppose $h_r=\alpha_q$ for some $q\in\N$. Use relation (6) to push $\alpha_q$ to the right of the $\beta$'s. There are five cases.
\be
\itt{Case 1} We use the relation $\alpha_i\beta_{j}=\tau^{\pm 1}$. Arguing as in Case (T) we are finished.
\itt{Case 2} We use the relation $\alpha_i\beta_{i\pm1}=\id_{[k]}$ for some $k\in\N\cup\{0\pm\}$, so $\alpha_qu_{III}=v_{III}$ for some $v_{III}\in\W_{III}$. Set $c_\pm=c'_\pm$ and $v_i=u_i$ for $i\in\{I,II\}$. We are finished.
\itt{Case 3} We use the relation $\alpha_i\beta_{i}=\delta_\pm$, so $\alpha_q u_{III}=\delta_\pm v_{III}$ for some $v_{III}\in \W_{III}$. Set $c_\pm=c'_\pm+1$, $c_\mp=c'_\mp$, and $v_i=u_i$ for $i\in\{I,II\}$. We are finished.
\itt{Case 4} $\alpha_q$ can be pushed all the way to the right of $u_{III}$ to obtain $\alpha_q u_{III}=v_{III}\alpha_p$ for some $p\in\N$ and $v_{III}\in\W_{III}$. Then necessarily $u_{II}=\tau^s$ for some $s\in\Z_{\geq 0}$, so we use relation (4) and \ref{additional} to push $\alpha_p$ to the right of the $\tau$'s. Hence we obtain $\alpha_p u_{II}=v_{II}\alpha_{k}$ for some $k\in\N$ and $v_{II}\in\W_{II}$. Set $c_\pm=c'_\pm$ and $v_I=\alpha_ku_I$. We are finished. 
\itt{Case 5} $\alpha_q$ can be pushed all the way to the right except for the last $\beta_i$. This means $\alpha_q u_{III}=v_{III} \alpha_i \beta_j$ for some $v_{III}\in\W_{III}$ where $\alpha_i\beta_j=\sigma_\pm$. Set $v_{II}=\sigma_\pm u_{II}$, $c_\pm=c'_\pm$, and $v_{I}=u_I$. We are finished.
\ee
\end{proof}

\begin{defn}\label{standardform}
If $w\in\Mor(\aD)$, a decomposition of $w$ as in Theorem \ref{standard} is called a standard form of $w$.
\end{defn}

\begin{rem}
It will be a consequence of Theorem \ref{iso} that a word $w\in\aD$ has a unique standard form.
\end{rem}

\section{The Isomorphism of Categories $\aD\cong\Atl$}\label{isocat}

\begin{prop} The following defines an involutive functor $F\colon \aD \to \Atl$:
\be
\itt{Objects}
$F([n])=[n]$ for all $n\in \N\cup \{0\pm\}$,
\itt{Morphisms}
\be
\item[(A)] 
Set $F(\alpha_i)=a_i$,
\item[(B)]
Set $F(\beta_i)=b_i$, 
\item[(T)]
Set $F(\tau)=t$, and
\item[(D)]
Set $F(\delta_+\in\aD(n,n))=(\id_{[n]},1,0)$ and $F(\delta_-\in\aD(n,n))=(\id_{[n]},0,1)$ for $n\in\N\cup\{0\pm\}$.
\ee
\ee
\end{prop}
\begin{proof}
We must check that $F(\id_{[n]})=\id_{[n]}$ for all  $n\in\N\cup\{0\pm\}$ and that $F$ preserves composition, but both these conditions follow from Theorem \ref{Atlrelations}. It is clear $*$ preserves the involution by Proposition \ref{Atlinvolution}.
\end{proof}

\begin{rem}
We construct a functor $G\colon \Atl\to \aD$ as follows: we create a function $G\colon \Atl\to \aD$ taking objects to objects (this part is easy as objects in both categories have the same names) and $\Atl(m,n)\to \aD(m,n)$ bijectively such that $F\circ G = \id_{\Atl}$. It will follow immediately that $G$ is a functor and $G\circ F=\id_{\aD}$.
\end{rem}

\begin{thm}\label{bijective} Let $m,n\in\N\cup\{0\pm\}$. Then $F_i=F|_{W_i(m,n)}\colon\W_i(m,n)\to \T_i(m,n)$ is bijective for all $i\in\{I,II,III\}$, i.e. there is a bijective correspondence between words of Type $i$ and Atl tangles of Type $i$ for all $i\in\{I,II,III\}$.
\end{thm}
\begin{proof} 
\itt{Type I} Note that $\im(F_I)\subset \T_I(m,n)$. We construct the inverse $G_I$ for $F_I$. Note that by Proposition \ref{a}, each $T\in\T_I(m,n)$ can be written uniquely as $T=W_r\cdots W_1$, which can further be expanded as
$$T=\underbrace{a_{i_p}\cdots a_{i_1}}_{W_r}\underbrace{a_{j_q}\cdots a_{j_1}}_{W_2} \cdots\underbrace{a_{k_r}\cdots a_{k_1}}_{W_1}$$ 
satisfying \ref{a}. Set
$$
G_I(T)=\alpha_{i_p}\cdots \alpha_{i_1}\cdots \alpha_{j_q}\cdots \alpha_{j_1}\alpha_{k_r}\cdots \alpha_{k_1}.
$$
It follows $F_I\circ G_I=\id$. Now by Proposition \ref{alpha}, every word of Type I can be written in this form. Hence we see $G_I$ is in fact the inverse of $F_I$
\itt{Type II} Obvious.
\itt{Type III}. From the Type I case and the involutions in $\aD$ and $\Atl$, we have the following bijections:
$$\T_{III}(m,n)\longleftrightarrow \T_{I}(n,m)\longleftrightarrow \W_I(n,m)\longleftrightarrow\W_{III}(m,n).$$
\end{proof}

\begin{defn}
We define $G\colon \Atl\to \aD$ as follows:
\itt{Objects} $G([n])=[n]$ for all $n\in\N\cup\{0\pm\}$.
\itt{Morphisms} First define $G(T,0,0)$ for a $T\in\T_i$ for $i\in\{I,II,III\}$ by the bijective correspondences given in Theorem \ref{bijective}. Then for an arbitrary Atl $(m,n)$-tangle $T$, define $G(T,0,0)$ by
$$G(T,0,0)=G(T_{III},0,0)\circ G(T_{II},0,0)\circ G(T_{I},0,0)$$
where $T_i$ for $i\in\{I,II,III\}$ is defined for $T$ as in the Atl Decomposition Theorem \ref{decomposition}. Finally, define $G(T,c_+,c_-)=\delta_+^{c_+}\delta_-^{c_-}G(T,0,0)$ for an arbitrary morphism $(T,c_+,c_-)\in \Mor(\Atl)$. Note that $G$ is well defined by the uniqueness part of \ref{decomposition}. 
\end{defn}

\begin{prop} If $T$ is an Atl $(m,n)$-tangle of Type $i$ for $i\in\{I,II,III\}$, then $F\circ G(T)=T$.
\end{prop}
\begin{proof}
This is immediate from the definition of $G$ and Theorem \ref{bijective}.
\end{proof}

\begin{cor}\label{injective}
$F\circ G=\id_{\Atl}$, so $G$ restricted to $\Atl(m,n)$ is injective into $\aD(m,n)$ for all $m,n\in\N\cup\{0\pm\}$.
\end{cor}
\begin{proof}
This follows immediately from Theorem \ref{decomposition} and the definition of $G$ as $F$ is a functor.
\end{proof}

\begin{prop}\label{surjective} $G$ restricted to $\Atl(m,n)$ is surjective onto $\aD(m,n)$.
\end{prop}
\begin{proof}
We have that every word $w\in\Mor(\aD)$ is equal to a word $\delta_+^{c_+}\delta_-^{c_-}w_{III}w_{II}w_{I}$ in standard form where $w_i$ is of Type $i$ for $i\in\{I,II,III\}$. By \ref{bijective} there are unique Atl tangles $T_{i}$ of Type $i$ such that $w_{i}=G(T_{i})$ for all $i\in\{I,II,III\}$. Set $T=T_{III}\circ T_{II}\circ T_{I}$, and note this decomposition into a composite of Atl tangles of Types I, II, and III is unique by \ref{decomposition}. It follows that 
$$G(T,c_+,c_-)=\delta_+^{c_+}\delta_-^{c_-}w_{III}w_{II}w_{I}=w$$
by the definition of $G$.
\end{proof}

\begin{thm}\label{iso}
$F\colon\aD\to \Atl$ is an isomorphism of involutive categories. Hence $\aD$ is a presentation of $\Atl$ via generators and relations.
\end{thm}
\begin{proof}
Obvious from Corollary \ref{injective} and Proposition \ref{surjective}.
\end{proof}

\begin{cor}\label{standardunique}
Each word $w\in\Mor(\aD)$ has a unique standard form.
\end{cor}
\begin{proof}
Each Atl tangle has a unique decomposition as $T_{III}\circ T_{II}\circ T_I$. Note $T_{III}^*$ and $T_I$ have unique decompositions as in Proposition \ref{a} which correspond under the isomorphism of categories to decompositions as in Proposition \ref{alpha}. We are finished.
\end{proof}

\section{The Annular Category from Two Cyclic Categories}\label{cyclicAtl}

\subsection{The Cyclic Category}\label{cyclic}
In this subsection, we recover Jones' result in \cite{MR1865703} that there are two copies of (the opposite of) the cyclic category $\ccD\op$ in $\aD\cong\Atl$. We will recycle the notation $t$ from Section 1. The definitions from this section are adapted from \cite{MR1217970}.

\begin{defn}
Let $\cAtl^+$ be the subcategory of $\Atl$ with objects $[n]$ for $n\in\N$ such that for $m,n\in\N$, $\cAtl(m,n)$ is the set of annular $(m,n)$-tangles with no loops, only shaded caps, and only unshaded cups. Let $\cAtl^-$ be the image of $\cAtl^+$ under the involution of $\Atl$, i.e. $\cAtl^-(m,n)$ is  the set of annular $(m,n)$-tangles with no loops, only unshaded caps, and only shaded cups.
\end{defn}

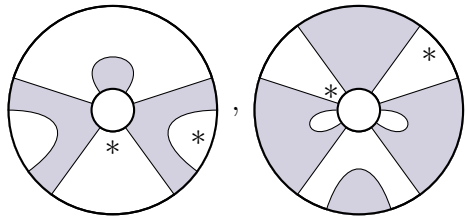
\begin{figure}[!ht]
\begin{tikzpicture}[annular]
	\clip (0,0) circle (2cm);
	\draw[ultra thick] (0,0) circle (2cm);	
	\node at (0,0)  [empty box] (T) {};
	\filldraw[shaded] (306:4cm)--(0,0)--(378:4cm);	
	\filldraw[shaded] (234:4cm)--(0,0)--(162:4cm);	
	\filldraw[shaded] (T.126) .. controls ++(126:11mm) and ++           (54:11mm) .. (T.54);
	\filldraw[unshaded] (180:4cm)--(180:2cm)  .. controls ++(0:11mm) and ++(36:11mm) .. (216:2cm) -- (216:4cm);
	\filldraw[unshaded] (0:4cm)--(0:2cm)  .. controls ++(180:11mm) and ++(144:11mm) ..  (324:2cm) -- (324:4cm);
	\draw[ultra thick] (0,0) circle (2cm);	
	\node at (0,0)  [empty box] (T) {};
	\node at (T.270) [below] {$*$};
	\node at (-15:2cm)[left] {$*$};
\end{tikzpicture}\hs,\hs
\begin{tikzpicture}[annular]
	\clip (0,0) circle (2cm);
	\draw[ultra thick] (0,0) circle (2cm);	
	\node at (0,0)  [empty box] (T) {};
	\filldraw[shaded] (306:4cm)--(0,0)--(378:4cm);	
	\filldraw[shaded] (234:4cm)--(0,0)--(162:4cm);
	\filldraw[shaded] (126:4cm)--(0,0)--(54:4cm);
	\filldraw[unshaded] (T.324) .. controls ++(324:8mm) and ++           (0:8mm) .. (T.0);
	\filldraw[unshaded] (T.216) .. controls ++(216:8mm) and ++           (180:8mm) .. (T.180);
	\filldraw[shaded] (288:4cm)--(288:2cm)  .. controls ++(108:11mm) and ++(72:11mm) .. (252:2cm) -- (252:4cm);

	\draw[ultra thick] (0,0) circle (2cm);	
	\node at (0,0)  [empty box] (T) {};
	\node at (T.122) [left] {$*$};
	\node at (45:1.9cm)[below] {$*$};
\end{tikzpicture}
\caption{Examples of morphisms in $\cAtl^+$ and $\cAtl^-$ respectively.}
\end{figure}

\begin{rem}
Clearly $\cAtl^+\cong \cAtl^-$.
\end{rem}

\begin{defn}
The opposite of the cyclic category $\ccD\op$ is given by
\itt{Objects} $[n]$ for $n\in\Z_{\geq 0}$ and
\itt{Morphisms} generated by
\begin{align*}
d_i\colon [n]&\longrightarrow [n-1]\hsp{for} i=0,\dots,n\hsp{where} n\geq 1\\
s_i\colon [n]&\longrightarrow [n+1]\hsp{for} i=0,\dots,n\hsp{where} n\geq 0\\
t\colon [n]&\longrightarrow [n]\hsp{where} n\geq 0
\end{align*}
subject to the relations
\be
\item[(1)] $d_id_j=d_{j-1}d_i$ for $i<j$.
\item[(2)] $s_is_j=s_{j+1}s_i$ for $i\leq j$,
\item[(3)] $\D d_i s_j = \begin{cases}
s_{j-1} d_i &\hsp{if} i<j\\
\id_{[n]} &\hsp{if} i=j,j+1\\
s_jd_{i-1}&\hsp{if}  i>j+1,
\end{cases}$
\item[(4)] $t^{n+1}=\id_{[n]}$,
\item[(5)] $d_it=td_{i-1}$ for $1\leq i\leq n$, and
\item[(6)] $s_it=ts_{i-1}$ for $1\leq i\leq n$.
\ee
\end{defn}

\begin{rem}\label{simplicial}
The opposite of the simplicial category $\sD\op$ is the subcategory of $\ccD\op$ generated by the $d_i$'s and the $s_i$'s subject to relations (1)-(3).
\end{rem}

\begin{rem}\label{cDadd}
Similar to Proposition \ref{additional}, we have the additional relations in $\ccD\op$ that $d_0t=d_n$ and $s_0t=t^2s_n$. 
\end{rem}

\begin{defn}\label{extra}
For $n\in \Z_{\geq 0}$, we define $s_{-1}\colon [n]\to[n+1]$ by $s_{-1}=ts_n$. This map is called the extra degeneracy.
\end{defn}
\begin{rem}
In \cite{MR1217970}, Loday names this map $s_{n+1}$. However, we will use the name $s_{-1}$ considering Proposition \ref{extradegeneracy}, Corollary \ref{beta1}, and the fact that if $R$ is a unital commutative ring, $A$ is a unital $R$-algebra, and  $C_\bullet$ is the cyclic $R$-module (see Section \ref{annularobjects}) arising from the Hochschild complex with coefficients in $A$, then $C_n=A^{\otimes n+1}$, and
\begin{align*}
s_{-1}(a_0\otimes\cdots \otimes a_n)&=1\otimes a_0\otimes\cdots\otimes a_n,\\
s_{i}(a_0\otimes\cdots \otimes a_n)&=a_0\otimes\cdots a_{i}\otimes 1\otimes a_{i+1}\otimes\cdots \otimes a_n\hsp{for} 0\leq i\leq n-1, \hsp{and}\\
s_{n}(a_0\otimes\cdots \otimes a_n)& =a_0\otimes\cdots\otimes a_n\otimes 1.
\end{align*}
\end{rem}

\begin{prop}\label{extradegeneracy}
The following additional relations hold for $s_{-1}\in\ccD\op(n,n+1)$:
\be
\item[(1)] $s_{-1}s_i=s_{i+1}s_{-1}$ for all $i\geq 0$,
\item[(2)] $\D d_is_{-1}=\begin{cases}
\id_{[n]} & \hsp{if} i=0\\
s_{-1}d_{i-1}  &\hsp{if} 1\leq i\leq n\\
t & \hsp{if} i=n+1, and
\end{cases}$
\item[(3)] $s_0t=ts_{-1}$.
\ee
\end{prop}
\begin{proof}
\item[(1)] Using relations (2) and (6), we get 
$$s_{-1}s_i = ts_{n+1}s_i=ts_is_{n}=s_{i+1}ts_{n}=s_{i+1}s_{-1}.$$
\item[(2)] Using Remark \ref{cDadd}, we have $d_0s_{-1}=d_0 ts_n = d_n s_n = \id_{[n]}$.
If $1\leq i\leq n$, then using relations (3) and (5), we have
$$
d_is_{-1} = d_i ts_n = td_{i-1} s_n = ts_{n-1}d_{i-1}=s_{-1}d_{i-1}.
$$
Finally, $d_{n+1}s_{-1}=d_{n+1}ts_n=td_ns_n=t\id_{[n]}=t$.
\item[(3)] Using Remark \ref{cDadd}, we have $s_0 t = t^2 s_n = ts_{-1}$.
\end{proof}
\begin{rem}
We may now add $s_{-1}$ to the list of generators of $\ccD\op$ after appropriately altering relations (3) and (6).
\end{rem}

\begin{prop}\label{cDstandard}
Suppose $w=h_r\cdots h_1$ is a word on $\ccD\op$ in $\ccD\op(m,n)$ for $m,n\in\Z_{\geq 0}$. Then there is a decomposition $w=w_{III}w_{II} w_{I}$ such that
\be
\item[(D)] $w_{I}=d_{i_a}\cdots d_{i_1}$ with $i_j>i_{j+1}$ for all $j\in\{1,\dots,a-1\}$.
\item[(T)] $w_{II}=t^k$ for some $k\geq 0$, and
\item[(S)] $w_{III}=s_{i_b}\cdots s_{i_1}$ with $i_j<i_{j+1}$ for all $j\in\{1,\dots,b-1\}$, 
\ee
\end{prop}
\begin{proof}
The proof is similar to Theorem \ref{standard}, but much easier. We proceed by induction on $r$. If $r=1$, the result is trivial. Suppose $r>1$ and the result holds for all words of shorter length. Apply the induction hypothesis to $w'=h_{r-1}\cdots h_1$ to get
$$
w'=u_{III}u_{II}u_I
$$
satisfying (1)-(3). There are three cases.
\be
\item[(T)] Suppose $h_r=t$. Set $w_I=u_I$. Use relation (6) and Remark \ref{cDadd} to push $t$ to the right of the $s_i$'s. Either it makes it all the way, or it disappears in the process. Define $w_{II}$ accordingly. Order the $s_i$'s using relation (2) to get $w_{III}$. We are finished.
\item[(D)] Suppose $h_r=d_i$. Use relation (3) to push $d_i$ to the right of the $s_j$'s. One of three possibilities occurs:
\be
\item[(1)] We only use the relation $d_i s_j=s_k d_l$. Thus we can push $d_i$ all the way to the right. Now push $d_i$ right of the $t$'s using relation (5) and Remark \ref{cDadd}. Order the $s_j$'s using relation (2) to get $w_{III}$, define $w_{II}$ in the obvious way, and reorder the $d_i$'s using relation (1) to get $w_I$. We are finished.
\item[(2)]  We use the relation $d_is_j=\id$, and $d_i$ disappears. Set $w_i=u_i$ for $i\in\{I,II\}$, and order the $s_j$'s using relation (2) to get $w_{III}$. We are finished.
\item[(3)] We use the relation $d_{n+1}s_{-1}=t$. We are now argue as in Case (T). We are finished.
\ee
\item[(S)] Suppose $h_r=s_i$. Order $s_iu_{III}$ using relation (2) to get $w_{III}$, and set $w_i=u_i$ for $i\in\{I,II\}$. We are finished.
\ee
\end{proof}

\begin{thm}
The following defines an injective functor $H^+\colon \ccD\op\to \aD$: 
\itt{Objects} $H^+([n])=[n+1]$ for $n\in\Z_{\geq 0}$, and
\itt{Morphisms} Let $n\in\Z_{\geq 0}$. 
\be
\item[(D)] For $j\in\{0,\dots,n\}$, set $H^+(d_j\in \ccD\op(n,n-1))=\alpha_{2j+1}\in \aD(n+1,n)$.
\item[(T)] Set $H^+(t\in \ccD\op(n,n))=\tau\in \aD(n+1,n+1)$.
\item[(S)] For $j\in\{0,\dots,n\}$, set $H^+(s_j\in \ccD\op(n,n+1))=\beta_{2j+2}\in \aD(n+1,n+2)$. 
\ee
\end{thm}
\begin{proof}
Clearly $H^+$ is a functor as the relations are satisfied. Injectivity follows immediately from Corollary \ref{standardunique} and Proposition \ref{cDstandard}. 
\end{proof}
\begin{rem}
Note that $H^+(s_{-1})=H^+(ts_n)=H^+(t)H^+(s_n)=\tau \beta_{2n+2}=\beta_{2n+4}\tau$.
\end{rem}
\begin{cor}
The image of $F\circ H^+\colon \ccD\op\to \Atl$ is $\cAtl^+$. Hence $\ccD\op\cong\cAtl^+$.
\end{cor}
\begin{proof}
It is clear $F\circ H^+$ is injective and lands in $\cAtl^+$ as all generators of $\ccD\op$ land in $\cAtl^+$. Surjectivity follows from Theorem \ref{decomposition}.
\end{proof}
\begin{cor}
A decomposition $w=w_{III}w_{II}w_I$ as in Proposition \ref{cDstandard} is unique.
\end{cor}

\begin{thm}
The following defines an injective functor $H^-\colon \ccD\op\to \aD$: 
\itt{Objects} $H^-([n])=[n+1]$ for $n\in\Z_{\geq 0}$, and
\itt{Morphisms} Let $n\in\Z_{\geq 0}$. 
\be
\item[(D)] For $j\in\{0,\dots,n\}$, set $H^-(d_j\in \ccD\op(n,n-1))=\alpha_{2j+2}\in \aD(n+1,n)$.
\item[(T)] Set $H^-(t\in \ccD\op(n,n))=\tau\in \aD([n+1],[n+1])$.
\item[(S)] For $j\in\{0,\dots,n\}$, set $H^-(s_j\in \ccD\op(n,n+1))=\beta_{2j+3}\in \aD(n+1,n+2)$. 
\ee
\end{thm}
\begin{proof}
Clearly $H^-$ is a functor as the relations are satisfied. Injectivity follows immediately from Corollary \ref{standardunique} and Proposition \ref{cDstandard}.
\end{proof}

\begin{rem}\label{beta1}
Note that $H^-(s_{-1})=H^-(ts_n)=H^-(t)H^-(s_n)=\tau \beta_{2n+3}=\beta_{1}$.
\end{rem}
\begin{cor}
The image of $F\circ H^-\colon \ccD\op\to \Atl$ is $\cAtl^-$. Hence $\ccD\op\cong\cAtl^-$.
\end{cor}

\begin{rem} $\cAtl^+$ and $\cAtl^-$ are exactly the two copies of $\ccD\op$ in $\Atl$ found by Jones in \cite{MR1865703}.
\end{rem}

\begin{cor}\label{opiso}
There is an isomorphism $\ccD\cong \ccD\op$.
\end{cor}
\begin{proof}
We have $\cAtl^-\cong\ccD\op\cong\cAtl^+$. Note the involution in $\Atl$ is an isomorphism $\cAtl^+\cong (\cAtl^-)\op$. The result follows.
\end{proof}

\subsection{Augmenting the Cyclic Category}\label{augment}

Recall from algebraic topology that the reduced (singular, simplicial, cellular) homology of a space $X$ is obtained by inserting an augmentation map $\varepsilon\colon C_0(X)\to \Z$ where $C_0(X)$ denotes the appropriate zero chains. In the language of the semi-simplicial category, we see that this is the same thing as looking at an augmented semi-simplicial abelian group, i.e., a functor from the opposite of the augmented semi-simplicial category, which is obtained from the opposite of the semi-simplicial category (see \ref{simplicial}) by adding an object $[-1]$ and the generator $d_0\colon [0]\to[-1]$ subject to the relation $d_id_j=d_{j-1}d_i$ for $i<j$.
\[
\xymatrix{
[-1] && [0]\ar[ll]_{d_0} && [1]\ar[ll]_{d_0,d_1} && [2]\ar[ll]_{d_0,d_1,d_2}&& \cdots\ar[ll]_{d_0,d_1,d_2,d_3}
}
\]
This immediately raises the question of how one should augment the opposite of the cyclic category. The surprising answer comes from the symmetry arising from the extra degeneracy $s_{-1}$. We should add two objects, $[+]$ and $[-]$, and maps $d_0\colon[0]\to[+]$ and $s_{-1}\colon[-]\to[0]$ subject to the relations $d_id_j=d_{j-1}d_i$ for $i<j$ and $s_i s_j=s_{j+1}s_i$ for $i\leq j$:
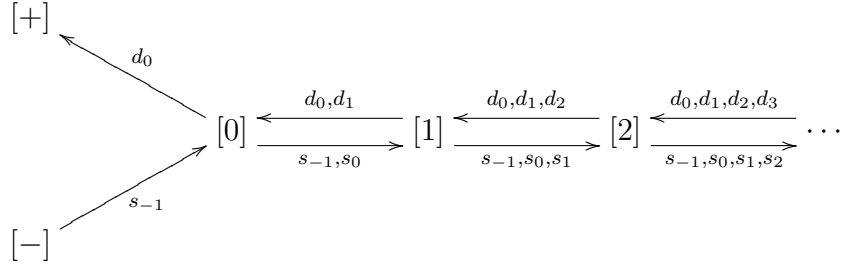
\[
\xymatrix{
[+]&&&&&\\
&& [0]\ar[llu]_{d_0}\ar@<-1ex>[rr]_{s_{-1},s_0}
&& [1]\ar@<-1ex>[ll]_{d_0,d_1}\ar@<-1ex>[rr]_{s_{-1},s_0,s_1} 
&& [2]\ar@<-1ex>[ll]_{d_0,d_1,d_2}\ar@<-1ex>[rr]_{s_{-1},s_0,s_1,s_2}
&& \cdots\ar@<-1ex>[ll]_{d_0,d_1,d_2,d_3}\\
[-]\ar[rru]_{s_{-1}}&&&&&&
}
\]
As $t\colon[0]\to[0]$ is the identity, we need not worry about the other relations.
Under the isomorphism $\ccD\op\cong\cAtl^+$ described in the previous subsection, these maps should be represented by the following diagrams:
\begin{figure}[!ht]
\begin{tikzpicture}[annular]
	\clip (0,0) circle (2cm);
	\draw[ultra thick] (0,0) circle (2cm);	
	\node at (0,0)  [empty box] (T) {};
	\node at (95:1.60cm) [right] {$*$};
	\filldraw[shaded] (T.-40) .. controls ++(-40:11mm) and ++           (40:11mm) .. (T.40);
	\draw[ultra thick] (0,0) circle (2cm);
	\node at (0,0)  [empty box] (T) {};
	\node at (T.70) [above] {$*$};		
\end{tikzpicture}$\hs,$
\begin{tikzpicture}[annular]
	\clip (0,0) circle (2cm);
	\draw[ultra thick][shaded] (0,0) circle (2cm);	
	\node at (0,0)  [empty box] (T) {};

	\filldraw[unshaded] (-45:3cm) --(-45:2cm) .. controls ++(135:7mm) and ++(-135:7mm) ..         (45:2cm) -- (45:3cm);	
	\node at (0:1.80cm) [left] {$*$};
	\draw[ultra thick] (0,0) circle (2cm);
	\node at (0,0)  [empty box] (T) {};
\end{tikzpicture}
\caption{Maps $d_0\colon[0]\to[+]$ and $s_{-1}\colon[-]\to[0]$}
\end{figure}

Note that these morphisms satisfy the shading convention of $\cAtl^+$ once we add $[0\pm]$ to the objects of $\cAtl^+$. We cannot use just one object as we would then violate the shading convention and closed loops would arise. We will denote the augmented opposite of the cyclic category by $\widetilde{\ccD\op}$. For our main result, we will also need to consider the augmented cyclic category $\widetilde{\ccD}$, which is just the category $\widetilde{\ccD\op}$ with the arrows switched.

\subsection{Pushouts of Small Categories}
Let $\Cat$ be the category of small categories. Note that pushouts exist in $\Cat$. 

\begin{defn}
Suppose $\AAA,\BB_1,\BB_2$ are small categories and $F_i\colon \AAA\to \BB_i$ for $i=1,2$ are functors. Then the pushout of the diagram
\[\xymatrix{
\AAA\ar[r]^{F_1}\ar[d]_{F_2} & \BB_1\\
\BB_2
}\]
is the small category $\CC$ defined as follows:
\itt{Objects} $\Ob(\CC)$ is the pushout in $\Set$ of the diagram
\[\xymatrix{
\Ob(\AAA)\ar[r]^{F_1}\ar[d]_{F_2} & \Ob(\BB_1)\ar@{..>}[d]^{G_1}\\
\Ob(\BB_2)\ar@{..>}[r]_{G_2} &\Ob(\CC)
}\]
This defines maps $G_i\colon \Ob(\BB_i)\to\Ob(\CC)$ for $i=1,2$.
\itt{Morphisms} 
For $X,Y\in\Ob(\CC)$, $\Mor(X,Y)$ is the set of all words  of the form $\varphi_n\circ\cdots\circ\varphi_1$ such that
\be
\item[(1)] $\varphi_i\in\Mor(\BB_1)\cup \Mor(\BB_2)$ for all $i=1,\dots,n$,
\item[(2)] the source of $\varphi_1$ is in $G_1^{-1}(X)\cup G_2^{-1}(X)$ and the target of $\varphi_n$ is in $G_1^{-1}(Y)\cup G_2^{-1}(Y)$,
\item[(3)] for all $i=1,\dots,n-1$, either
\be
\item[(i)] the target of $\varphi_i$ is the source of $\varphi_{i+1}$, or
\item[(ii)] the target of $\varphi_i$ is $Z_i\in \im(F_j)\subseteq\BB_j$ for some $j\in\{1,2\}$, and the source of $\varphi_{i+1}$ is in $F_k(F_j^{-1}(Z_i))$ where $k\neq j$.
\ee
\ee
subject to the relation $F_1(\psi)=F_2(\psi)$ for every morphism $\psi\in\Mor(\AA)$. 
\end{defn}

\begin{nota}
In the sequel, we will need to discuss $\widetilde{\ccD}$, the augmented cyclic category. In order that no confusion can arise, we will add a $*$ to morphisms to emphasize the fact that they compose in the opposite order. For example, we have generators $d_i^*$ satisfying the relation $d_j^*d_i^*=d_i^*d_{j-1}^*$ for $i<j$. 
\end{nota}

\begin{defn}
Define the small category/groupoid $\TT$ by
\itt{Objects} $[n]$ for $n\in\Z_{\geq 0}\cup\{\pm\}$
\itt{Morphisms} Generated by $t\colon[n]\to[n]$ subject to the relation $t^{n+1}=\id_{[n]}$ for $n\in\Z_{\geq 0}$.
\end{defn}

\begin{defn}
Let $\PO$ be the pushout in $\Cat$ of the following diagram:
\[\xymatrix{
\TT\ar[r]^{F_1}\ar[d]_{F_2} & \widetilde{\ccD\op}\\
\widetilde{\ccD}
}\]
where $F_i([n])=[n]$ for $n\in\Z_{\geq 0}\cup[\pm]$ for $i=1,2$ and $F_1(t)=t$ and $F_2(t)=(t^*)^{-1}=(t^{-1})^*$. Note that if $\ccD\op$ has generators $d_i,s_i,t$ and $\ccD$ has generators $d_i^*,s_i^*,t^*$, then $\PO$ is the category given by
\itt{Objects} $[n]$ for $n\in\Z_{\geq 0}\cup\{\pm\}$ and
\itt{Morphisms} generated by 
\begin{align*}
d_0\colon [0]&\to [+]\hsp{and} s_{-1}^*\colon [0]\to[-]\\
s_{-1}\colon [+]&\to [0]\hsp{and} d_0^*\colon[-]\to[0]\\
d_i,s_{i-1}^*\colon [n]&\longrightarrow [n-1]\hsp{for} i=0,\dots,n\hsp{where} n\geq 1\\
s_i,d_{i+1}^*\colon [n]&\longrightarrow [n+1]\hsp{for} i=-1,\dots,n\hsp{where} n\geq 0\\
t\colon [n]&\longrightarrow [n]\hsp{where} n\geq 0
\end{align*}
subject to the relations
\be
\item[(1)] $d_id_j=d_{j-1}d_i$ and $s_i^*s_j^*=s_{j-1}^*s_i^*$ for $i<j$,
\item[(2)] $s_is_j=s_{j+1}s_i$ and $d_i^*d_j^*=d_{j+1}^*d_i^*$ for $i\leq j$,
\item[(3)] $\D d_i s_j = \begin{cases}
s_{j-1} d_i &\hsp{if} i<j\\
\id_{[n]} &\hsp{if} i=j,j+1\\
s_jd_{i-1}&\hsp{if}  i>j+1
\end{cases}$
and 
$s_{i-1}^*d_j^*=\begin{cases}
d_{j-1}^*s_{i-1}^* &\hsp{if} i<j-1\\
\id_{[n]}  &\hsp{if} i=j,j+1\\
d_j^*s_{i-2}^* &\hsp{if} i>j+1,
\end{cases}$
\item[(4)] $t^{n+1}=\id_{[n]}$,
\item[(5)] $d_it=td_{i-1}$ for $1\leq i\leq n$ and $s_i^* t=ts_{i-1}^*$ for $0\leq i\leq n$, and
\item[(6)] $s_it=ts_{i-1}$ for $0\leq i\leq n$ and $d_i^*t=td_{i-1}^*$ for $1\leq i\leq n$.
\ee
Note that $t=(t^*)^{-1}$ as $\PO$ is the pushout, so $t^*$ does not appear in the above list.
\end{defn}

\begin{rem}
Note that $\PO$ is involutive using the obvious involution as hinted by the $*$-notation.
\end{rem}

\begin{defn}
Let $\PO(\delta_+,\delta_-)$ be the small category obtained from $\PO$ by adding generating morphisms  $\delta\pm\colon [n]\to[n]$ for all $n\in\Z_{\geq 0}\cup\{\pm\}$ which commute with all other morphisms. The maps $\delta_\pm$ are called the coupling constants.
\end{defn}

\begin{rem}
Note that $\PO(\delta_+,\delta_-)$ is involutive if we define $(\delta_\pm)^*=\delta_\pm$.
\end{rem}

\begin{thm}\label{pushoutAtl} $\aD$ is isomorphic to the category $\QQ$ obtained from $\PO(\delta_+,\delta_-)$ with the additional relations
\item[(1)] $\D d_i s_j^*=\begin{cases}
s_{j-1}^*d_i &\hsp{if} i<j\\
s_j^* d_{i+1}^* &\hsp{if} j>i
\end{cases}$
\item[(2)]  $\D d_i d_j^*=\begin{cases}
d_{j-1}^*d_i &\hsp{if} i<j\\
\delta_- &\hsp{if} i=j
\end{cases}$
\item[(3)] $\D s_i^* s_j=\begin{cases}
s_{j-1}s_i^* &\hsp{if} i<j\\
\delta_+ &\hsp{if} i=j
\end{cases}$
\end{thm}
\begin{proof}
Define a map $\Psi\colon\aD\to\QQ$ by 
\itt{Objects} Define $\Psi([0\pm])=[\pm]$. For $n\geq 1$, define $\Psi([n])=[n-1]$.
\itt{Morphisms} We define $\Psi$ on primitive morphisms:
\be
\item[(A)] 
Define $\Psi(\alpha_1\in\aD(1,0+))=d_0\in\QQ(0,+)$ and $\Psi(\alpha_2\in\aD(1,0-))=s_{-1}^*\in\QQ(0,-)$. For $n\geq 2$, define
$$\Psi(\alpha_i\in\aD(n,n-1))=\begin{cases}
s^*_{(i-3)/2}\in\QQ(n-1,n-2) &\hsp{if $i$ is odd} \\
d_{(i-2)/2}\in\QQ(n-1,n-2) &\hsp{if $i$ is even.} 
\end{cases}$$
\item[(B)] 
Define $\Psi(\beta_1\in\aD(0+,1))=s_{-1}\in \QQ(-,0)$ and $\Psi(\beta_2\in\aD(0-,1))=d_0^*\in\QQ(+,0)$. For $n\geq 1$, define
$$\Psi(\beta_i\in\aD(n,n+1))=\begin{cases}
s_{(i-3)/2}\in\QQ(n-1,n) &\hsp{if $i$ is odd} \\
d^*_{(i-2)/2}\in\QQ(n-1,n) &\hsp{if $i$ is even.} 
\end{cases}$$
\item[(T)] For $n\geq 1$, define $\Psi(\tau\in\aD(n,n))=t\in\QQ(n-1,n-1)$.
\item[(D)] Define $\Psi(\delta_\pm)=\delta_{\pm}$.
\ee
One checks $\Psi$ is a well defined isomorphism by showing the relations match up.
\end{proof}

\begin{rems}
\item[(1)] The above relations are called the coupling relations.
\item[(2)] Usually we study representations of $\ccD$ and $\aD$ in abelian categories and the coupling constants are multiplication by scalars. These scalars can be built into the coupling relations in our abelian category without first defining $\PO(\delta_+,\delta_-)$. Hence an annular object in an abelian category (see Section \ref{annularobjects}) is obtained from the pushout of two cyclic objects over a $\TT$-object and then quotienting out by the coupling relations.
\item[(3)] Another way to skip passing to $\PO(\delta_+,\delta_-)$ is to take the linearization of all our categories over some unital commutative ring $R$ (make the morphism sets $R$-modules) and choose  scalars $\delta_\pm$ for the coupling relations.
\end{rems}

\section{Annular Objects}\label{annularobjects}

\begin{defn}
An annular object in an arbitrary category $\CC$ is a functor $\aD\to \CC$. A cyclic object is a functor $\ccD\op\to\CC$. If $\CC$ is an abelian category and $X_\bullet$ is an annular, respectively cyclic, object, we replace $X_\bullet(\tau\in\aD(n,n))$ with $(-1)^{n-1}X_\bullet(\tau)$, respectively we replace $X_\bullet(t\in\ccD\op(n,n))$ with $(-1)^{n}X_\bullet(t)$, to account for the sign of the cyclic permutation. 
\end{defn}

\begin{rems}
Each annular object has two restrictions to cyclic objects.
\end{rems}

\begin{nota}
Usually such a functor is denoted with a bullet subscript, e.g. $X_\bullet$. If $X_\bullet$ is such a functor, we will use the following standard notation:
\be
\item[(1)] $X_\bullet([n])=X_n$ for $n\in\Z_{\geq 0}$ and $X_{\bullet}([0\pm])=X_\pm$ where applicable.
\item[(2)] $X_\bullet(\varphi)=\varphi$, i.e. we will use the same notation for the images of the morphisms in the category $\CC$. 
\ee
\end{nota}


\begin{note}
For an annular object in an abelian category, relations (4), (5), and (6) become 
\be
\item[(4')]
$\alpha_i\tau = -\tau\alpha_{i-2}$ for $i\geq 3$,
\item[(5')]
$\beta_i \tau = -\tau\beta_{i-2}$ for $i\geq 3$, and
\item[(6')] if $\alpha_i\beta_j\colon [n]\to [n]$ with $n\geq 2$ and $(i,j)=(1,2n+2),(2n+2,1)$, then $\alpha_1\beta_{2n+2}=(-1)^{n-1}\tau^{-1}$ and $\alpha_{2n+2}\beta_{1}=(-1)^{n-1}\tau$.
\ee
Proposition \ref{additional} becomes
\be
\item[(1')]
$\alpha_1\tau = (-1)^{n-1}\alpha_{2n-1}$ and $\alpha_2\tau = (-1)^{n-1}\alpha_{2n}$
\item[(2')] $\tau\beta_{2n+1}=(-1)^n\beta_1$, $\tau\beta_{2n+2}=(-1)^{n}\beta_{2}$, and
\item[(3')] $\beta_1 \tau = (-1)^{n-1}\tau^2\beta_{2n-1}$ and $\beta_2 \tau = (-1)^{n-1}\tau^2\beta_{2n}$.
\ee
\end{note}

\begin{note}
For a cyclic object in an abelian category, relations (5) and (6) become
\be
\item[(5')]
$d_it=-td_{i-1}$ for $i\geq 1$ and  
\item[(6')]
$s_it=-ts_{i-1}$ for $i\geq 1$.
\ee
Following Remark \ref{cDadd}, we have
\be
\item[(i)]
$d_0t = (-1)^n d_n$ and
\item[(ii)] $s_0 t = (-1)^{n}t^2s_{n}$.
\ee
Definition \ref{extra} becomes $s_{-1}=(-1)^{n+1}ts_n$. Parts (2) and (3) of Proposition \ref{extradegeneracy} become
\be
\item[(2')] $d_{n+1}s_{-1}=(-1)^nt$ and
\item[(3')]  $s_0t=-ts_{-1}$.
\ee
\end{note}

\begin{rem}
The necessity of this sign convention becomes apparent in calculations with Connes' boundary map (see \ref{cyclichomology} and \ref{connes}).
\end{rem}

\begin{defn}
Let $\CC$ be an involutive category and suppose $X_\bullet\colon \aD\to \CC$ is an annular object in $\CC$. Then $X_\bullet^*\colon \aD\to \CC$ is also an annular object in $\CC$ where
\be
\itt{Objects} $X_\bullet^*([n])=X_n$ for all $n\in\N\cup\{0\pm\}$, and
\itt{Morphisms} $X_{\bullet}^*(w)=X_\bullet(w^*)^*$ for all $w\in\Mor(\aD)$.
\ee
If $\CC$ is abelian, then $X_\bullet^*$ still satisfies the sign convention.
\end{defn}

\begin{rem}
The representation theory of $\Atl$ was studied extensively by Graham and Lehrer in \cite{MR1659204} and Jones in \cite{MR1929335}.
In Definition/Theorem 2.2 in \cite{0902.1294}, Peters gives a good summary of  the case of an annular $C^*$-Hilbert module where $\delta_\pm$ is given by multiplication by $\delta> 2$.
\end{rem}

\subsection{Homologies of Annular Modules}\label{homology}
As the semi-simplicial, simiplicial, and cyclic categories live inside $\aD$, we can define Hochschild and cyclic homologies of annular objects in abelian categories. We will focus on annular modules and leave the generalization to an arbitrary abelian category to the reader. Fix a unital commutative ring $R$.

\begin{defn}
Given a semi-simplicial $R$-module $M_\bullet$, define the Hochschild boundary $b\colon M_n\to M_{n-1}$ for $n\geq 1$ by
$$
b=\sum\limits^n_{i=0} (-1)^id_i.
$$
The Hochscild homology of $M_\bullet$ is
$$
HH_n(M_\bullet,b)=\ker(b)/\im(b)
$$
for $n\geq 0$, where we set $M_{-1}=0$, and $b\colon M_0\to M_{-1}$ is the zero map.
\end{defn}

\begin{rem}
As an annular $R$-module is a semi-simplicial $R$-module in two ways, we will have two Hochschild boundaries.
\end{rem}

\begin{defn}
Suppose $X_\bullet$ is an annular $R$-module. Let $X_\bullet^\pm$ be the cyclic object obtained from $X_\bullet$ by restricting $X_\bullet$ to $G(\cAtl^\pm)$. For $n\geq 1$, define 
$$HH_n^\pm(X_\bullet) = HH_{n-1}^\pm(X_\bullet^\pm).$$
\end{defn}

\begin{rem}
The Hochschild boundaries of $X_\bullet^\pm$ for $n\geq 2$ are
$$
b_+=\sum\limits^{n-1}_{i=0} (-1)^i \alpha_{2i+1}
\hsp{and}
b_-=\sum\limits^{n-1}_{i=0} (-1)^i \alpha_{2i+2}.
$$
\end{rem}

\begin{defn}
The above definition does not take into account $X_\pm$. We may define the reduced Hochschild homology by looking at the corresponding augmented cyclic objects (see Subsection \ref{augment}). Define $b_\pm\colon X_1\to X_\pm$ by  $b_+=\alpha_1\colon X_1\to X_+$ and $b_-=\alpha_2\colon X_1\to X_-$. Define the reduced Hochschild homology of $X_\bullet$ as follows:
\begin{align*}
\widetilde{HH}^\pm_n(X_\bullet)&=HH_{n}^\pm(X_\bullet)\hsp{for} n\geq 2,\\
\widetilde{HH}_1^\pm(X_\bullet)&=\ker(b_\pm)/\im(b_\pm),\hsp{and}\\
\widetilde{HH}_0^\pm(X_\bullet)&=\coker(b_\pm)
\end{align*}
\end{defn}

\begin{rem}
The content of the next proposition was found by Jones in \cite{MR1865703}.
\end{rem}

\begin{prop}
Let $X_\bullet$ be an annular $R$-module. Then for all $n\geq 1$,
\begin{align*}
\beta_{1}b_++b_+\beta_{1}&=\delta_+ \id_{X_n}\hsp{and}\\
\beta_{2}b_-+b_-\beta_{2}&=\delta_- \id_{X_n},
\end{align*}
and when $n=\pm$, 
\begin{align*}
b_+\beta_1&=\delta_+ \id_{X_+}\hsp{and}\\
b_-\beta_2&=\delta_- \id_{X_-}.
\end{align*}
\end{prop}
\begin{proof}
This follows immediately from relation 6.
\end{proof}

\begin{cor}
If $\delta_\pm$ is multiplication by an element of $R^\times$, the group of units of $R$, then $\widetilde{HH}_n^\pm(X_\bullet)=0$ for all $n\geq 0$.
\end{cor}

\begin{cor}
Let $N\subset M$ be an extremal, finite index $II_1$-subfactor, and let $X_\bullet$ be the annular $\C$-module given by its tower of relative commutants (see \cite{math/9909027}, \cite{MR1929335}). Then $\widetilde{HH}^\pm_n(X_\bullet)=0$ for all $n\geq 0$.
\end{cor}

\begin{ex}[$TL_\bullet(\Z,0)$]
When $\delta_\pm\notin R^\times$, we can have non-trivial homology. For example, for $n\in\N\cup\{0\pm\}$, let $TL_n(\Z,0)$ be the set of $\Z$-linear combinations of planar $n$-tangles with no input disks and no loops (adjust the definition of an annular $n$-tangle so that there is no $D_1$). The action of $T\in\aD(m,n)$ on $S\in TL_m(\Z,0)$ is given by tangle composition $F(T)\circ S$ with the additional requirement that if there are any closed loops, we get zero. We then extend this action $\Z$-linearly. Then  $HH^\pm_n(TL_\bullet(\Z,0))\neq 0$ for all $n\geq 0$. In fact, the class of the planar $n$-tangle with only shaded, respectively unshaded, cups is a nontrivial element in $HH^\pm_n(TL_\bullet(\Z,0))$ respectively. Clearly all such tangles are in $\ker(b_\pm)$. However, it is only possible to get an even multiple of this tangle in $\im(b_\pm)$. If a shaded region is capped off by an $\alpha_i$ to make a cup, there must be two ways of doing so.
\begin{figure}[!ht]
\begin{tikzpicture}[annular]
	\clip (0,0) circle (2cm);
	\node at (135:1.80cm) {$*$};
	\filldraw[shaded] (60:2cm) .. controls ++(240:11mm) and ++(300:11mm) ..    (120:2cm) -- (120:3cm);
	\filldraw[shaded] (-30:2cm) .. controls ++(150:11mm) and ++(210:11mm) ..   (30:2cm) -- (30:3cm);
	\filldraw[shaded] (150:2cm) .. controls ++(-30:11mm) and ++(30:11mm) ..     (210:2cm) -- (210:3cm);
	\filldraw[shaded] (240:2cm) .. controls ++(60:11mm) and ++(120:11mm) ..    (300:2cm) -- (300:3cm);

	\draw[ultra thick] (0,0) circle (2cm);
\end{tikzpicture}
\caption{Representative for a nontrivial element in $HH^+_4(TL_\bullet(\Z,0))$}
\end{figure}
Using MAGMA \cite{MR1484478}, the author has calculated the first few ($+$) reduced Hochschild homology groups of $TL_\bullet(\Z,0)$ to be
\begin{align*} 
\widetilde{HH}^+_0&=\widetilde{HH}^+_1=\Z,\\
HH^+_2&=HH^+_3=\Z/2,\\
HH^+_4&=HH^+_5=\Z/6,\hsp{and}\\
HH^+_6&=HH^+_7=\Z/2\oplus\Z/2.
\end{align*}
For $n\geq 2$, the class of the tangle described above contributes a copy of $\Z/2\Z$. The question still remains whether this parity continues. 
\end{ex}

\begin{defn}\label{cyclichomology}
Given a cyclic $R$-module $C_\bullet$, define the cyclic bicomplex $BC_{**}(C_\bullet)$ be the bicomplex
\[
\xymatrix{
\ar[d]_b & \ar[d]_b & \ar[d]_b & \ar[d]_b\\
C_3\ar[d]_b & C_2\ar[d]_b\ar[l]_B & C_1\ar[d]_b\ar[l]_B & C_0\ar[l]_B\\
C_2\ar[d]_b & C_1\ar[d]_b\ar[l]_B & C_0\ar[l]_B\\
C_1\ar[d]_b & C_0\ar[l]_B \\
C_0
}
\]
where $b$ is the Hochschild boundary obtained by looking at the corresponding semi-simplicial $R$-module, 
$B=(1-t)s_{-1}N\colon C_n\to C_{n+1}$ is Connes' boundary map, and 
$$N=\sum\limits^{n}_{i=0} t^i.$$ 
Recall $s_{-1}=(-1)^{n+1}ts_n$ is the extra degeneracy. The cyclic homology of $C_\bullet$ is given by $HC_n(C_\bullet)=H_n(\Tot(BC_{**}(C_\bullet)))$. 
\end{defn}

\begin{rem}\label{connes}
In order for $BC_{**}(C_\bullet)$ to be a bicomplex, we need $b^2$, $B^2$, and $bB+Bb$ to equal zero. While the first two are trivial, we must use Loday's sign convention to get the third. Setting 
$$
b'=\sum\limits^{n-1}_{i=0} (-1)^i d_i\colon C_n\to C_{n-1},
$$
we have $b(1-t)=(1-t)b'$, $b's_{-1}+s_{-1}b'=\id$, and $b'N=Nb$, so
$$
bB+Bb= b (1-t)s_{-1} N +(1-t)s_{-1} N b = (1-t)(b's_{-1}+s_{-1}b')N = (1-t)N=0.
$$
Without this sign convention, we no longer have $bB+Bb=0$.
\end{rem}

\begin{defn}
Suppose $X_\bullet$ is an annular $R$-module. Then $X_\bullet$ becomes a cyclic module in two ways, so we have two cyclic homologies to study. For $n\geq 1$, define $HC^\pm_n(X_\bullet)=HC_{n-1}(X^\pm_\bullet)$.
\end{defn}

\begin{rem}
For $n\geq 1$, $B_\pm\colon X_n\to X_{n+1}$ is given by
\begin{align*}
B_+&=(-1)^n (1-\tau)(\tau\beta_{2n})\sum\limits^{n-1}_{i=0}\tau^i=(-1)^n(1-\tau) (\beta_{2n+2}\tau)\sum\limits^{n-1}_{i=0} \tau^i\\
&=(-1)^n(1-\tau) \beta_{2n+2}\sum\limits^{n-1}_{i=0} \tau^i\hsp{and}\\
B_-&=(-1)^n(1-\tau)\beta_1\sum\limits^{n-1}_{i=0} \tau^i
\end{align*}
as the two extra degeneracies for $G(\cAtl^\pm)$ are $(-1)^n\tau\beta_{2n}$ and $(-1)^n\beta_1$ respectively.
\end{rem}

\begin{cor}
If $\delta_\pm$ is multiplication by an element of $R^\times$, the group of units of $R$, then $HC_n^\pm(X_\bullet)=R$ for all odd $n\geq 1$ and  $HC_n^\pm(X_\bullet)=0$ for all even $n\geq 2$.
\end{cor}

\begin{cor}
Let $N\subset M$ be an extremal, finite index $II_1$-subfactor, and let $X_\bullet$ be the annular $\C$-module given by its tower of relative commutants. Then $HC_n^\pm(X_\bullet)=\C$ for all odd $n\geq 1$ and  $HC_n^\pm(X_\bullet)=0$ for all even $n\geq 2$.
\end{cor}

\begin{ex}
Once again using MAGMA \cite{MR1484478}, the author has calculated the first few ($+$) cyclic homology groups of $TL_\bullet(\Z,0)$ to be
\begin{align*} 
HC^+_1&=\Z\\
HC^+_2&=\Z/2\\
HC^+_3&=\Z/2\oplus\Z\\
HC^+_4&=\Z/2\oplus\Z/6\\
HC^+_5&=\Z/2\oplus\Z/6\oplus\Z,\hsp{and}\\
HC^+_6&=\Z/2\oplus\Z/2\oplus\Z/2\oplus\Z/6.
\end{align*}
\end{ex}

\bibliographystyle{amsalpha}
\bibliography{bibliography}
\end{document}